\documentclass[11pt,a4paper,reqno,final]{article}
\usepackage{template}

\DeclareMathOperator\Arm{Arm}
\DeclareMathOperator\E{E}
\let\P\relax
\DeclareMathOperator\P{P}
\DeclareMathOperator{\poisson}{Poisson}
\DeclareMathOperator{\Count}{count}

\usepackage[backend=bibtex,giveninits=true,maxcitenames=12,maxbibnames=9]{biblatex}
\renewcommand{\t}{\tau}

\newcommand{\Z}{\mathbb{Z}}

\newcommand{\R}{\mathbb{R}}
\newcommand{\N}{\mathbb{N}}

\newcommand{\eqd}{\stackrel{d}{=}}

\makeatletter
\newcommand\xleftrightarrow[2][]{%
  \ext@arrow 9999{\longleftrightarrowfill@}{#1}{#2}}
\newcommand\longleftrightarrowfill@{%
  \arrowfill@\leftarrow\relbar\rightarrow}
\makeatother

  \newcounter{iconst}

\begin{document}

\title{Percolation phase transition on planar spin systems}

\date{\today}
\author{
  Caio Alves
  \thanks{Email: \ \texttt{caiotmalves@gmail.com}; \ Alfréd Rényi Institute of Mathematics, Budapest, 1053 Hungary.}
  \and
  Gideon Amir
  \thanks{Email: \ \texttt{gidi.amir@gmail.com}; \ Bar-Ilan University, 5290002, Ramat Gan, Israel.}
  \and
  Rangel Baldasso
  \thanks{Email: \ \texttt{r.baldasso@math.leidenuniv.nl}; \ Mathematical Institute, Leiden University, P.O. Box 9512, 2300 RA Leiden, The Netherlands.}
  \and
  Augusto Teixeira
  \thanks{Email: \ \texttt{augusto@impa.br}; \ IMPA, Estrada Dona Castorina 110, 22460-320 Rio de Janeiro, RJ - Brazil.}
}

\maketitle

\begin{abstract}
  In this article we study the continuity and sharpness of the phase transition for percolation models defined on top of planar spin systems.
    The two examples that we treat in detail concern the Glauber dynamics for the Ising model and a Dynamic Bootstrap process.
    For both of these models we prove that their phase transition is continuous and sharp, providing also quantitative estimates on the two point connectivity.
    The techniques that we develop in this work can be applied to a variety of different percolation models based on spin-flip dynamics.
    We also discuss some of the problems that can be tackled in a similar fashion.
    
  \medskip


\end{abstract}

\section{Introduction}
~

\par Since the introduction of percolation as a mathematical model in \cite{PSP:2048852}, the subject has undergone impressive developments.
These include a better understanding of its subcritical and supercritical phases \cite{menshikov1986coincidence, aizenman1987sharpness, grimmett1990supercritical, duminil2016new}.
Moreover, although the critical phase has remained evasive to the best of the community's efforts, there are two remarkable exceptions that are now well understood: the planar model \cite{kesten1980, smirnov2001critical} and the high dimensional case \cite{hara1990mean}.

In parallel to the above mentioned advances, big developments were also made in the study of dependent percolation.
These works extend some of the results that were already known in the Bernoulli case, both in arbitrary dimensions \cite{meester1996continuum, Szn09} as well as in the plane \cite{zbMATH05064645, vandenberg2011, tassion2016crossing}.

This paper continues this trend by looking at examples of dependent percolation on the plane and establishing the continuity and the sharpness of their phase transitions.
We treat two examples in detail: Glauber dynamics of the Ising model and dynamical bootstrap percolation.

\vspace{4mm}

Let us start by introducing the first model we consider.
For this, fix a density $\rho \in [0, 1]$ and let us define the initial distribution of our spin system.
This will be denoted as $\sigma^\rho_0(x)$, for $x \in \mathbb{Z}^2$, and will be i.i.d.\ taking value $1$ with probability $\rho$ and $-1$ with probability $1 - \rho$.

Having established the initial configuration, the process evolves according to the Glauber dynamics of the Ising model at inverse temperature $\beta \in [0, \infty]$, see Section~\ref{s:model_and_notations} for a precise definition.
This defines the process $\sigma^\rho_t(x)$, for all $x \in \mathbb{Z}^2$ and $t \geq 0$.

In this article we study the percolation configuration obtained at a fixed finite time $\t \geq 0$ as we vary $\rho$ from zero to one.
In \eqref{e:monotone} below we show that, for fixed $\t$, this process is monotone in $\rho$, meaning that we can couple all $\sigma^{\rho}_{\t}(x)$ in a way that if $\rho \geq \rho'$ then $\sigma^\rho_{\t}(x) \geq \sigma^{\rho'}_{\t}(x)$ for every $x \in \mathbb{Z}^2$. Let $\P_{\rho}$ denote the distribution of $\sigma_{\t}^{\rho}$.
We then look at $(\sigma^\rho_{\t}(x))_{x \in \mathbb{Z}^2}$ as a dependent percolation model on the plane as the value of $\rho$ varies.
See Figures~\ref{f:heatmaps1} and~\ref{f:heatmaps2}  for a simulation of the process for various values of~$\rho$.

Given a configuration $\sigma: \mathbb{Z}^2 \to \{-1, 1\}$, we write $A \xleftrightarrow{+} B$ for the event that the sets $A, B \subseteq \mathbb{Z}^2$ can be connected by a nearest-neighbor path of $+1$ spins.
We can then define the critical point

\begin{figure}[ht]
    \centering
    \includegraphics[width=10cm]{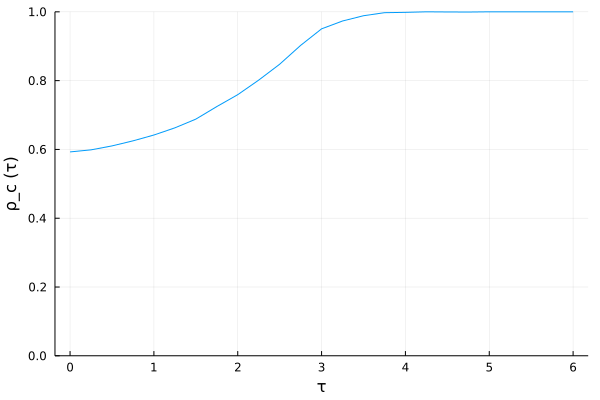}
    \caption{Simulation showing how the critical value of percolation varies as $\t$ increases in a high-temperature regime, with $\beta = 0.1$. Using a union-find algorithm, $\rho_c(\t)$ was estimated as the first parameter such that a left-right crossing was observed inside a box $150$ by $150$. $300$ simulation runs were made for each value of $\t$, which was discretized into steps of size $0.25$. The occurrence of percolation quickly becomes impossible, as the process resembles Bernoulli site percolation with parameter $0.5$ no matter the initial conditions.}
    \label{f:simulation2}
  \end{figure}

\begin{equation}
  \label{e:rho_c}
  \rho_c(\tau) = \inf \Big\{ \rho \in [0, 1]; \P_\rho \big[ 0 \xleftrightarrow{+} \infty \big] > 0 \Big\}.
\end{equation}
We recall that the dependence in $\t$ in the infimum above is via $\P_{\rho}$, the distribution of $\sigma^{\rho}_{\t}(x)$. We will omit this dependence in $\tau$ in the probability in order to ease notations.

Inspecting the simulations in Figures~\ref{f:simulation1} and~\ref{f:simulation2}, it becomes clear that for small values of $\t$ there is a phase transition for the percolation of $\sigma^\rho_{\t}$ as we vary $\rho$, while for larger values of $\t$ and high temperature (small $\beta$), the system stays always in a single phase.

Our first result makes the above observation precise, by showing the existence or absence of phase transitions in $\rho$ as we vary $\t$ and $\beta$.

\nc{c:t_small}
\nc{c:t_large_decay}
\begin{theorem}[Phase transition for Glauber dynamics]
  \label{t:transition}
  Under one of the following conditions:
  \begin{enumerate}
  \item $\beta = \infty$,
  \item or if $\beta \in (0, \infty)$ and $\t \in [0, \uc{c:t_small}]$ (for some $= \uc{c:t_small} = \uc{c:t_small}(\beta) > 0$),
  \end{enumerate}
  the system undergoes a non-trivial phase transition, meaning that $\rho_c(\t) \in (0, 1)$.

  On the other hand, there exists~$\beta_0 > 0$ such that, for $\beta \in [0,\beta_0)$ and $\t$ large enough (depending on $\beta$), the configuration $\sigma^\rho_{\t}$ is subcritical uniformly on $\rho \in [0,1]$.
  More precisely, there exists $\epsilon \in (0,1)$ and $\uc{c:t_large_decay} = \uc{c:t_large_decay}(\beta) > 0$ such that, for every $\rho \in [0, 1]$ and $\t \geq \uc{c:t_large_decay}$,
  \begin{equation}
    \P_\rho \big[ B_n \xleftrightarrow{+} \partial B_{2n} \big] \leq e^{ - \uc{c:t_large_decay} n^{\epsilon} }, \text{ for every $n \geq 1$}.
  \end{equation}
In particular, the set in~\eqref{e:rho_c} is empty.
\end{theorem}

\begin{figure}[ht]
    \centering
    \includegraphics[width=10cm]{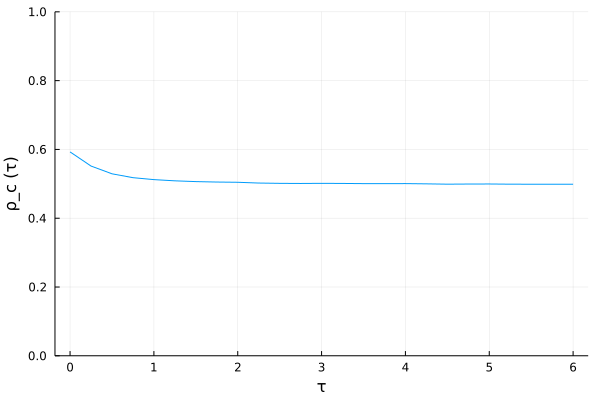}
    \caption{Simulation showing how the critical value of percolation varies as $\t$ increases in a low-temperature regime, with $\beta = 10$. Simulated in the same manner as \ref{f:simulation2}. The simulations hint at the self-duality of the process, as the critical parameters seem to converge to $0.5$.}
    \label{f:simulation1}
  \end{figure}

Our next theorem is the most central part of this work and it gives a detailed description of the phase transition that happens as we vary $\rho$.
Informally speaking, this theorem gives detailed descriptions of the three phases of the system: critical, subcritical and supercritical.

\nc{c:critical}
\nc{c:sub_critical}
\nc{c:super_critical}
\begin{theorem}[Phase transition for Glauber dynamics]
  \label{t:phases}
  For any value of $\beta \in [0, \infty]$, if $\t \geq 0$ is such that $\rho_c(\t) \in (0, 1)$, then
  \begin{equation}
    \label{e:continuous}
    \P_\rho \big[ 0 \overset{+}\leftrightarrow \infty \big] \text{ is continuous as a function of $\rho$}.
  \end{equation}
  Moreover,
  \begin{equation}
    \label{e:critical}
    \P_{\rho_c(\t)} \big[ B_n \xleftrightarrow{+} \partial B_{2n} \big] \in (\uc{c:critical}, 1 - \uc{c:critical}), \text{ for all $n \geq 1$},
  \end{equation}
  for some constant $\uc{c:critical} > 0$ (that depends only on $\beta$ and $\t$).
  On the other hand, for all $\rho < \rho_c(\t)$ and $\epsilon>0$, there exists $\uc{c:sub_critical} = \uc{c:sub_critical}(\rho, \epsilon, \beta, \t) > 0$ such that
  \begin{equation}
    \label{e:fast_to_zero}
    \P_\rho \big[ B_n \xleftrightarrow{+} \partial B_{2n} \big] \leq \exp \big\{ - \uc{c:sub_critical} n/\log^\epsilon(n) \big\}, \text{ for every $n \geq 1$},
  \end{equation}
  and if $\rho > \rho_c(\t)$, there exists $\uc{c:super_critical} = \uc{c:super_critical}(\rho, \epsilon, \beta, \t) > 0$ such that
  \begin{equation}
    \label{e:fast_to_one}
    \P_\rho \big[ B_n \xleftrightarrow{+} \partial B_{2n} \big] \geq 1 - \exp \big\{ - \uc{c:super_critical} n/\log^\epsilon(n) \big\}, \text{ for every $n \geq 1$}.
  \end{equation}
\end{theorem}

\begin{remark}
  Here are a few observations that complement the above stated results.
  \begin{enumerate}
  \item Regarding Theorem~\ref{t:transition}, we conjecture that, for any $\beta<\beta_{c}$, the critical paramenter for the magnetization of the Ising model, there is no phase transition, provided the parameter $\t$ is large enough. On the other hand, for $\beta> \beta_{c}$, we expect that the systems undergoes a non-trivial phase transition for all $\t>0$.
  \item We have stated our main Theorem~\ref{t:phases} for the Glauber dynamics of the Ising model.
    However, the techniques we employ can give the same results for some other models of spin-flip dynamics on the plane.
    In order to emphasize this generality we treat a second model in Section~\ref{s:elliptic_bootstrap}.
    In Section~\ref{s:open_problems} we explain for what types of other models our techniques should be expected to be helpful.
  \item Although the simulations suggest that $\rho_c(\t)$ is monotone as a function of $\t$, this is currently unknown as the model lacks microscopic monotonicity in $\t$. The possible continuity of $\rho_c(\t)$, also suggested by the simulations, is also an interesting and currently unknown question.
  \item Note that our results apply to the zero temperature regime $\beta = \infty$.
    In this case there is indeed a non-trivial phase transition for all $\tau \geq 0$ and the three distinct phases can be observed.
  \item Our results can also be applied in the (self-dual) triangular lattice, where they imply that $\rho_{c}(\t)=\frac{1}{2}$, for all $\t>0$, if one considers symmetric dynamics, such as Glauber dynamics or majority dynamics (see~\cite{alves2019sharp}).
  \end{enumerate}
\end{remark}

\begin{figure}[ht]
  \centering
  \includegraphics[width=10cm]{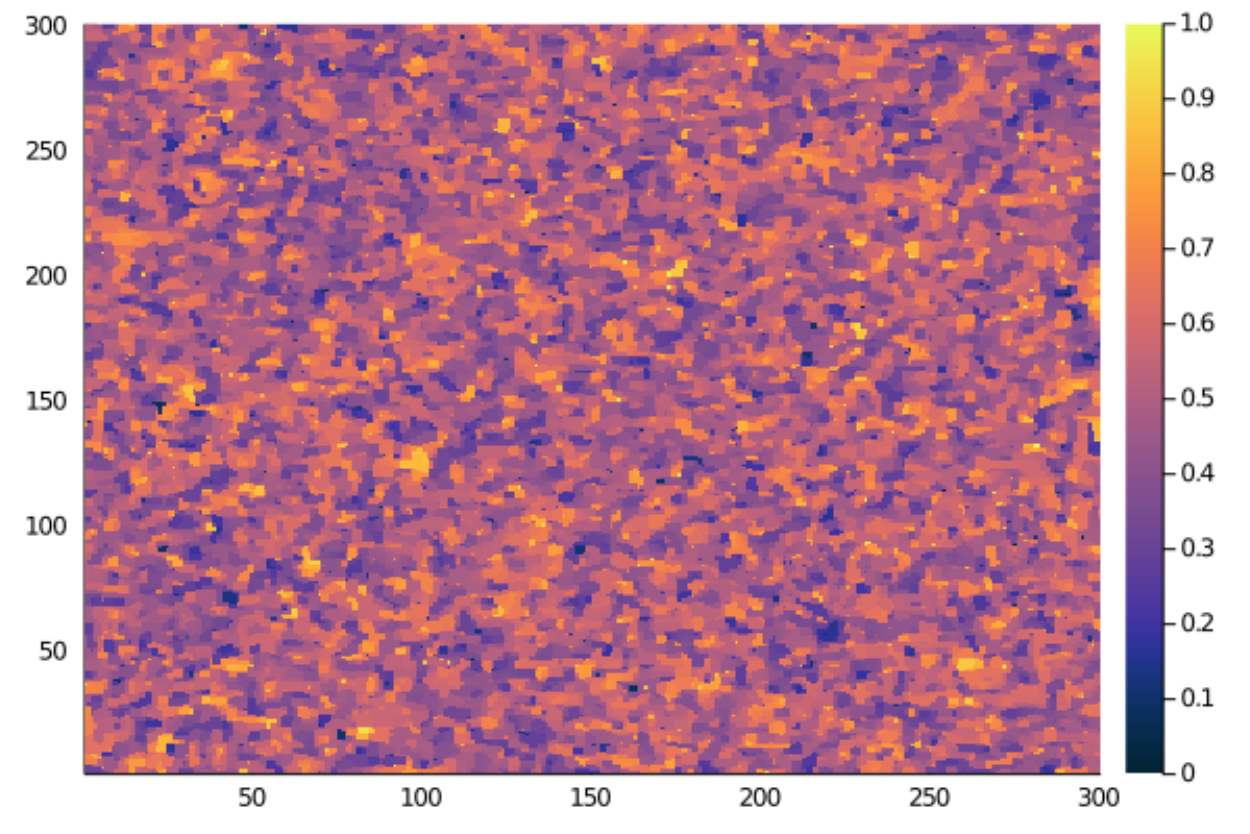}
  \caption{In this simulation, each site starts with i.i.d. $\mathrm{Uniform}[0, 1]$ variables as internal parameters, and each time the Poisson clock of a site rings, it either copies the parameter of one of its neighbors, or sets its internal parameter as $0$ or $1$, all in accordance with the Glauber dynamics. In this way, the level sets of the function which takes each site to its associated parameter at a given time $\t$ are distributed as the Glauber dynamics percolation. The above image is a heatmap representation of this function, giving the required minimum value of the parameter~$\rho$ so that the state of a site is $+1$ in a box $300 \times 300$ at time $\t  = 6$ for $\beta = 3$. It shows how at low temperature the internal parameters of the sites become close to $1/2$.}
  \label{f:heatmaps1}
\end{figure}

\begin{figure}[ht]
  \centering
  \includegraphics[width=10cm]{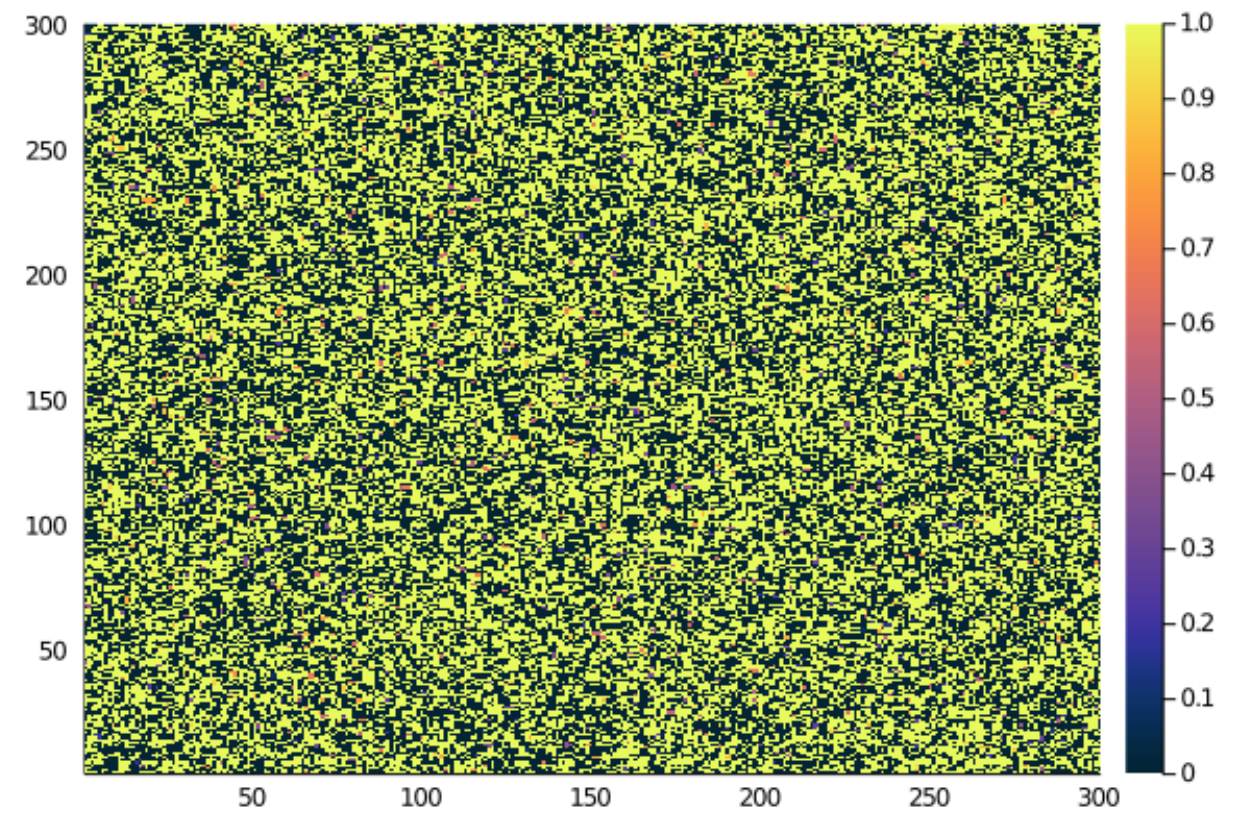}
  \caption{Heatmap representing the required minimum value of the parameter~$\rho$ so that the state of a site is $+1$ in a box $300 \times 300$ at time $\t  = 6$ for $\beta = 0.1$. Simulated in the same way as Figure~\ref{f:heatmaps1}. It shows how for high temperature and long times, the initial internal parameters of the sites do not matter, as the process converges to Bernoulli percolation with parameter $1/2$.}
  \label{f:heatmaps2}
\end{figure}

Let us now introduce the second model considered here, that we call \textit{Dynamic Elliptic Bootstrap}.
Once again, this system is constructed through the evolution of a particle system, but in this case the phase transition will occur as a function of time.
This choice of model was made in order to exemplify that our techniques can work on quite different models, even if their parameters play a very different role on the distribution of the configuration.
However, it is important to notice that monotonicity is essential throughout our proofs.

At time zero, all sites have opinion -1 and as time passes they can flip their opinion to 1 and become frozen afterwards.
Each site changes opinion from $-1$ to $1$ with rate given by a monotone increasing function of its neighborhood. Furthermore, we assume that this function is invariant under ($\pi/2$)-rotations, has finite range, and a positive lower bound $\chi > 0$. The last assumption implies that each vertex, independent of its neighborhood configuration, flips with a positive rate. The precise definition of the model can be found in Section~\ref{s:elliptic_bootstrap}. Let $\sigma_{t}(x)$ denote the opinion of $x$ at time $t$, and denote by $\P_{t}$ the distribution of the configuration $\sigma_{t}$.

In this case, we examine the percolation threshold as a function of time by introducing, analogously to~\eqref{e:rho_c},
\begin{equation}
t_{c} = \inf \Big\{ t \in [0,\infty): \P_{t}\big[ 0 \xleftrightarrow{+} \infty \big] > 0 \Big\}.
\end{equation}
A simple argument relying on stochastic domination by Bernoulli percolation implies that
\begin{display}
  $t_{c}$ is non-trivial, meaning that $t_c \in (0, \infty)$.
\end{display}
Our main result here is analogous to Theorem~\ref{t:phases}, where we once again obtain descriptions of the three distinct phases this process undergoes.

\nc{c:bep_critical}
\nc{c:bep_sub_critical}
\nc{c:bep_super_critical}
\begin{theorem}[Sharp threshold for elliptic bootstrap]
  \label{t:phases_bootstrap}
  There exists a constant $\uc{c:bep_critical} > 0$ such that, for any $n \geq 1$,
  \begin{equation}
    \label{e:critical_bootstrap}
    \P_{t_{c}} \big[ B_n \xleftrightarrow{+} \partial B_{2n} \big] \in (\uc{c:bep_critical}, 1 - \uc{c:bep_critical}).
  \end{equation}
 Furthermore, for all $t < t_{c}$ and $\epsilon>0$, there exists $\uc{c:bep_sub_critical} = \uc{c:bep_sub_critical}(t, \epsilon) > 0$ such that
 \begin{equation}
   \label{e:t_small}
    \P_{t} \big[ B_n \xleftrightarrow{+} \partial B_{2n} \big] \leq \exp \big\{ - \uc{c:bep_sub_critical} n/\log^\epsilon(n) \big\}, \text{ for every $n \geq 1$},
  \end{equation}
  and, if $t > t_{c}$, there exists $\uc{c:bep_super_critical} = \uc{c:bep_super_critical}(t, \epsilon) > 0$ such that
  \begin{equation}
    \label{e:t_large}
    \P_\rho \big[ B_n \xleftrightarrow{+} \partial B_{2n} \big] \geq 1 - \exp \big\{ - \uc{c:bep_super_critical} n/\log^\epsilon(n) \big\}, \text{ for every $n \geq 1$}.
  \end{equation}
\end{theorem}

\begin{remark}
The choice of this model as an additional example allows us point out the flexibility of the technique we employ here and also its limitations. First of all, in elliptic bootstrap the order parameter over which sharp-threshold behavior is established plays a fundamentally different role, namely the time in an interacting particle system. Second, the analysis of this model provides more insight on what are the essential assumptions one needs to verify in order to carry on the analysis: strong monotonicity with respect to the order parameter, Aizenman-Grimmett-like estimates, and Russo's formula. In this case, Russo's formula can be obtained directly from calculations involving the generator of this Feller process, while the pivotality relations require a more in-depth understanding of the graphical construction of the process.
\end{remark}

\vspace{4mm}

\paragraph{Related models.} Previous results that more closely relate to this work have dealt with Voronoi percolation~\cite{zbMATH05064645,tassion2016crossing}, contact-process percolation~\cite{vandenberg2011}, confetti percolation~\cite{hirsch2015harris} and Boolean percolation~\cite{ATT18}.

The above works begin to establish a general framework for dealing with planar dependent percolation in a systematic way.
Nonetheless, there is still no general result that can treat large classes of models using the same proof strategy.

Roughly speaking, the proof of continuity and sharpness for percolation models on the plane typically involves three steps: finite-size criteria, Russo-Seymour-Welsh estimates and sharpness.
While the first two steps have reached a very high degree of generality already, notably with the very recent development in \cite{tassion-koeler20}, the sharpness of the phase transition still requires model dependent arguments.

One of the objectives of this article is to further the techniques to prove sharpness of the phase transition, consequently describing two models that were not previously studied.
Moreover, we emphasize several of the limitations of the currently known techniques and dedicate a whole section to stating open problems that explore the field and its shortcomings.

\paragraph{Continuity in two dimensions.} The first result showing continuity of the phase transition for a percolation process was proved by Kesten in~\cite{kesten1980} on $\mathbb{Z}^2$, which in particular proved that the critical point of bond percolation in~$\Z^2$ is~$1/2$.
In his paper, Kesten relied on the work of Harris~\cite{harris1960lower} and the techniques from Russo~\cite{Russo} and Seymour and Welsh~\cite{SeymourWelsh}, which give lower bounds for the probabilities of rectangles being crossed by the percolation process in the hard direction in terms of the crossing probabilities of squares. These bounds are at the heart of the study of critical percolation of many planar percolation processes, and are generally called the \emph{Russo-Seymour-Welsh theory}, or RSW for short.

Continuity and sharpness have been established for dependent percolation models since then.
An important example in this direction is given by \cite{zbMATH05064645}, for Voronoi percolation.
Similar techniques relying on RSW-type results and differential inequalities based on influence bounds were also used in the study of confetti percolation \cite{hirsch2015harris} and the contact process~\cite{vandenberg2011}.

The field of dependent planar percolation received renewed attention after a major improvement on RSW bound in \cite{tassion2016crossing}, a framework which was later applied in other works such as \cite{ATT18}. Very recently, another major development in the field was proved in \cite{tassion-koeler20}, which obtained RSW bounds under minimal hypotheses, such as positive association (FKG) and invariance of the underlying percolation measure under reflections and ($\pi/2$)-rotations.
Recent progress has also been made for phase transitions in models that lack the positive association provided by FKG, see \cite{muirhead2020phase}.

Continuity and sharpness of Bernoulli percolation on ``almost planar'' graphs were also obtained in \cite{DST16}, concerning the product~$\Z^2 \times \{1, \dots, n\}$ for some~$n \in \N$.
We can also cite \cite{beffara2012self}, which studies the planar random cluster model, as an example of the study of sharp thresholds on planar dependent percolation processes.

\paragraph{Sharpness in higher dimensions.}
As one moves away from the planar case, the continuity of the phase transition remains a very important open problem, with the exception of some progress that has been made for high dimensional independent percolation \cite{hara1990mean}.
However the question of sharpness is being actively researched.

The works of \cite{menshikov1986coincidence,aizenman1987sharpness} and \cite{grimmett1990supercritical} have established the sharpness of Bernoulli percolation in any dimensions, although they strongly rely on the independence of the underlying environment. For dependent percolation models, some recent techniques have been very successful and are worth mentioning individually.

First, let us mention the OSSS inequality that has been introduced in \cite{OSSS} and has proven to be very useful in proving sharpness of phase transitions for dependent percolation models in any dimensions.
Examples include FK-type models \cite{D-CRT19} and the subcritical phase of the occupied set in Boolean percolation \cite{duminil2020subcritical}.
Although this technique is very promising (and used in the current article), applying the OSSS to a percolation model still involves some arguments that are very model dependent, as discussed below.

Using the OSSS inequality requires one to construct the desired process in terms of independent Bernoulli (or uniform) random variables.
Moreover, this construction needs to satisfy a variety of properties, like monotonicity and small influence in order to be successful. In particular, the implementation of this technique used in the present work is not suitable for higher dimensions. We elaborate further on this topic in Remark~\ref{r:higher_dimensions}.

Finally, let us mention the ``interpolation technique'' for proving sharpness of phase transitions for dependent processes in any dimensions.
It has been successfully applied for the Level Sets of Gaussian Free Fields
\cite{duminil2020equality} and other Gaussian Fields \cite{severo2021sharp}.
Despite being an early development, this method is very promising and is expected to yield interesting sharpness results for other models such as Random Interlacements.
However, interpolation techniques make use of sprinkling in the order parameter, so that even in the planar case interpolation techniques cannot prove continuity of the phase transition.

In our understanding, there are still several obstacles that the community must overcome in order to obtain a general theory of sharp thresholds and continuity of the phase transition for dependent percolation processes.
Though the techniques for proving a \emph{finite-size criterion} and establishing a \emph{RSW theory} are now very general, the final key consisting in \emph{sharp threshold results} from \emph{differential inequalities} is very model-dependent.
Finding correct analogies for Russo's formula and obtaining related differential inequalities exhibiting this threshold behavior continues to be, for the time being, a labor that must be repeated for each distinct model.
However it is becoming clear that patterns are starting to emerge as this study progresses.

\paragraph{Proof overview.} As alluded to above, our proof, like the proof in the seminal paper of Kesten \cite{kesten1980}, relies on three pillars:

\begin{itemize}
\item A finite-size criterion, meaning that if the probability that the cluster crosses a large rectangle is sufficiently small (resp., sufficiently close to $1$), then we are in the subcritical (resp., supercritical) phase.
  The proof is classical and relies on a model-agnostic multiscale renormalization scheme, as well as a decoupling inequality, see Lemma~\ref{l:decouple} and the discussion below \eqref{e:t_c}.

\item Russo-Seymour-Welsh theory, meaning lower bounds of crossing probabilities of rectangles in terms of the crossing probabilities of squares. We use the general result obtained in \cite{tassion-koeler20}, which only require that the model satisfy some symmetry conditions and positive association, see \eqref{e:fkg1} and below \eqref{e:t_c}.

\item Sharp thresholds, meaning that crossing probabilities of large rectangles are either very close to~$0$, or very close to~$1$, and that the intermediate interval of the parameter in which this probability is far away from both~$0$ and~$1$ collapses to a point (the critical point) as the size of these rectangles grow, if the process undergoes a phase transition.
\end{itemize}

The first two pillars are not new: our main contribution is showing sharp thresholds for the models considered here. Our argument relies on randomized algorithms and the OSSS inequality~\cite{OSSS}. This is not straightforward though, since we have to deal with the randomness coming from the evolution together with the randomness of the initial condition, which is the focus of our analysis. To avoid degeneracies that might appear due to badly behaved realizations of the dynamics, we provide an alternative graphical construction that has a richer structure and allows us to work on the quenched setting, by comparing it to its annealed counterpart. Once properties of this construction are established, the core of our argument lies in a way of relating the influences of different types of variables present in this alternative construction of the processes we study. This main inequality between the probabilities of these variables being pivotal is inspired by the work of Aizenman and Grimmett~\cite{AG91}.

\paragraph{Overview of the paper.} In Section~\ref{s:model_and_notations}, we define the Glauber dynamics percolation model, introducing notation and stating preliminary results, such as Russo's formula and a finite-size criterion. In Section~\ref{s:phase_transition_of_glauber}, we prove Theorem~\ref{t:transition}, studying the existence of phase transitions for different values of the parameter~$\beta$. In Section~\ref{s:proof_idea}, we give an overview of the OSSS inequality and its uses in proving sharp thresholds. In Section~\ref{s:proof_ingredients}, we introduce another graphical construction of the model which allow us to compare the process conditioned on the clock ticks to the unconditioned process. In Section~\ref{s:AG}, we finish the proof of Theorem~\ref{t:phases} by proving an inequality relating different types of pivotalities. In Section~\ref{s:elliptic_bootstrap}, we prove analogous results for the dynamic elliptic bootstrap percolation model.
We dedicate Section~\ref{s:open_problems} to listing various open problems that we expect to progress in the near future.

\paragraph{Acknowledgments.}
CA was supported by the Noise-Sensitivity Everywhere ERC Consolidator Grant 772466 and the DFG Grant SA 3465/1-1.
GA was supported by the Israel Science Foundation Grant 957/20.
RB has counted on the support of the Mathematical Institute of Leiden University.
During this period, AT has also been supported by grants ``Projeto Universal'' (406250/2016-2) and ``Produtividade em Pesquisa'' (304437/2018-2) from CNPq and ``Jovem Cientista do Nosso Estado'', (202.716/2018) from FAPERJ.

\section{Main model and auxiliary results}\label{s:model_and_notations}
~
\par We start with the general notation. We will consider~$\Z^2$ with the nearest-neighbor graph structure. We write~$x \sim y$ in case~$x, y \in \Z^2$ are neighbors, and in general let~$|x - y|_1$ denote the~$\ell_1$-distance between $x$ and $y$. The $\ell_\infty$-ball with radius~$n > 0$ and center at~$x \in \Z^2$ will be denoted by~$B(x, n)$. Given a set~$A \subset \Z^2$, we denote its complement by~$A^c$; and define its \textit{internal boundary} as
\begin{equation}
  \partial A := \big\{ x \in A; \text{ there exists $y \in A^c$ such that $x \sim y$} \big\}.
\end{equation}
We also define the \textit{external boundary} of~$A$ as~$\partial_{ext}A := \partial A^c$, and its cardinality (which may be infinite if not otherwise specified) as~$|A|$.
Note that $|\partial B(x, n)| = \max\{ 1, 8n \}$.
Given $A, B \subset \Z^2$, define their \textit{mutual distance}~$\dist(A, B)$ as the infimum of the $\ell_\infty$-distance between their elements.

We denote the \textit{uniform distribution} over an interval~$[a, b]\subset\R$ by~$U[a, b]$, the \textit{exponential distribution} with parameter~$\lambda > 0$ by~$\Exp(\lambda)$, the \textit{Poisson distribution} with parameter~$\lambda > 0$ by~$\poisson(\lambda)$, and the \textit{Bernoulli distribution} with parameter~$p \in [0, 1]$ by~$\Ber(p)$. The symbol~$\N$ will denote the natural numbers, and~$\N_0$ will denote the set of natural numbers \emph{including} $0$.

Let us now define the model we are interested in. We construct in standard fashion a collection of time-indexed random variables associated to each point of~$\Z^2$
  \begin{equation}
    \big(
      \sigma_t(x); \, t \geq 0, x \in \Z^2
    \big)
  \end{equation}
such that~$\sigma_t(x) \in \{1, -1\}$, for every~$t \geq 0$ and~$x \in \Z^2$. Given~$\rho \in [0, 1]$, we let~$(\sigma_0(x))_{x \in \Z^2}$ be an i.i.d.\ family of random variables such that
  \begin{equation}
    \sigma_0(x) \eqd \rho \delta_1 + (1 - \rho) \delta_{-1},
  \end{equation}
where~$\delta_y$ denotes the Dirac measure at~$y \in \R$. For each $x \in \mathbb{Z}^2$, define now an independent Poisson point process with rate~$1$ on~$\mathbb{R}_+$, whose marks, to which we will informally refer as time marks or clock ticks, will be denoted by $\mathcal{P}^x = \{T_0^x, T_1^x, \dots\}$. To each pair~$(x, T^x_i)$, $i \in \N_0$, we associate an independent mark
  \begin{equation}
    \label{eq:time_mark_uniform_rv}
    R^x_i \eqd U[0,1].
  \end{equation}
We write $x \overset{t}\rightarrow y$ if there exists a sequence of nearest-neighbor sites
  \begin{equation}
    x = x_0 \sim \dots \sim x_k = y
  \end{equation}
  and a sequence of marks $T^{x_0}_{i_0} < \dots < T^{x_k}_{i_k} < t$ in the respective Poisson point processes.
  We will need the following result, which will help us measure how much the state of the process in one site can influence the state of the process in a distant site:
  \nc{c:light}
  \begin{lemma}
    \label{l:light}
    For every $t > 0$, there exists $\uc{c:light}(t) > 0$ ($\exp \{2^{14}t(1+t^{3}) \log(8t)\}$ suffices) such that, for all $x \neq y \in \mathbb{Z}^2$,
    \begin{equation}
      \mathbb{P} \big[ x \overset{t}\rightarrow y \big]
      \leq
      \uc{c:light}(t)
      \exp \bigg\{  - \frac{1}{4} |x - y|_1 \log (|x - y|_1) \bigg\},
    \end{equation}
    where $\mathbb{P}$ here denotes the probability measure associated with the Poisson point processes $(\mathcal{P}^x)_{x \in \Z^2}$.
  \end{lemma}

We postpone the proof of this lemma to the Appendix.

Heuristically, we will define the process so that, at each time mark associated to a site, said site will update its state according to the Ising process distribution conditioned on the state of its four neighbors. For $\beta \in [0, \infty]$ and~$z_1, \dots, z_4 \in \{-1, 1\}$, we denote
\begin{equation}
    \label{eq:update_definition}
  S_\beta(z_1, \dots, z_4) :=
    \frac{
      \exp \big\{ \beta \sum_{i = 1}^4 z_i \big\}
     }{
      \exp \big\{ \beta \sum_{i = 1}^4 z_i \big\} +
      \exp \big\{ - \beta \sum_{i = 1}^4 z_i \big\},
     }
\end{equation}
where the case~$\beta = \infty$ should be interpreted as the appropriate limit of the above expression, that is,
\begin{equation*}
  S_\infty(z_1, \dots, z_4) :=
    \begin{cases}
      1/2, \text{ if } \sum_{i = 1}^4 z_i = 0;
      \\
      0, \text{ if } \sum_{i = 1}^4 z_i < 0;
      \\
      1, \text{ if } \sum_{i = 1}^4 z_i > 0.
    \end{cases}
\end{equation*}
We then define the \emph{update function}~$g_\beta : \{-1, 1\}^4 \times [0, 1] \to \{-1, 1\}$ as
\begin{equation}
  \label{eq:gbeta}
  g_\beta (z_1, \dots, z_4, u) :=
    \begin{cases}
      1, \qquad & \text{if }u \leq S_\beta(z_1, \dots, z_4) ;
      \\
      -1, & \text{otherwise}.
    \end{cases}
\end{equation}
We can now finally define the model. Given~$x \in \Z^2$, we let $\sigma_t(x) = \sigma_0(x)$ for all $t \in [0, T^x_1)$; and for each~$i \geq 1$ and $t \in [T^x_i, T^x_{i+1})$, we define
\begin{equation}
\label{e:gupdate}
  \sigma_t(x) =
    g_\beta \big( (\sigma_{T^{x}_{i}}(y))_{y \sim x}, R^x_i \big).
\end{equation}
In principle it is not clear that this process is well defined, since in order to define~$\sigma_t(x)$, one might need to know the state of the process at the neighbors of~$x$. However, Lemma~\ref{l:light} guarantees that, almost surely, one only needs to look at the state of finitely many vertices at time~$0$ in order to determine~$\sigma_t(x)$.

From now on we will fix~$\t \geq 0$ and~$\beta \in~[0, \infty]$. Given these parameters, we will study the model $(\sigma_{\t}(x); x \in \Z^2)$ as a site percolation process on~$\mathbb{Z}^2$. This means we will investigate connectivity properties of the set of points with positive magnetization~$\{x \in \Z^2; \sigma_{\t}(x) = 1\}$. We say that~$A$ is $+$-connected to~$B$ in~$C$ and write
\begin{equation}
  \Big\{ A \xleftrightarrow{+,\,\, C} B \Big\}
\end{equation}
for the event where there exists a nearest-neighbor path $x_0, \dots, x_n \in C \subset \Z^2$ starting at~$A \subset \Z^2$ and ending at~$B \subset \Z^2$ such that
\begin{equation}
  \sigma_{\t}(x_0) = \dots = \sigma_{\t}(x_n) = +1.
\end{equation}
If~$C = \Z^2$, we omit it from the notation. We also write
\begin{equation}
  \Big\{ A \xleftrightarrow{+} \infty \Big\}
    := \bigcap_{n \geq 0} \Big\{ A \xleftrightarrow{+} \partial B(0, n) \Big\}.
\end{equation}
If either~$A$ or~$B$ is a singleton~$\{x\}$, $x \in \Z^2$, we write~$x$ instead of~$\{x\}$ in the notation for the above events. We say that two points in~$\Z^2$ are $*$-neighbors if the~$\ell_\infty$-distance between them is~$1$. We analogously define the event where~$A$ is~$(-, *)$-connected to~$B$ in~$C$, writing
\begin{equation}
\Big\{ A \xleftrightarrow{-,\,\, *,\,\, C} B \Big\}
\end{equation}
when there is a $*$-connected path, that is, a path whose consecutive sites are~$*$-neighbors, starting at~$A$ and ending at~$B$ whose associated states at time~$\t$ are all~$-1$.

We denote by~$\P_\rho$ the probability measure associated to the process with initial density of positive states given by~$\rho \in [0, 1]$, and by~$\E_\rho$ the associated expectation. We can then define the probability that the origin percolates
\begin{equation}
  \theta(\rho) := \P_\rho \Big( 0 \xleftrightarrow{+} \infty \Big),
\end{equation}
and the \textit{critical parameter}
\begin{equation}
  \rho_c := \sup \big\{ \rho \in [0,1] ; \, \theta(\rho) = 0 \big\}.
\end{equation}
We can analogously define~$\rho^{*}_{c}$, the supremum of the set of parameters~$\rho$ such that with positive probability the origin is~$(-, *)$-connected to~$\infty$.

It is easy to see that
\begin{display}
  \label{e:monotone}
  $\theta(\rho)$ is monotone with respect to~$\rho$,
\end{display}
given the uniform coupling between starting densities and the fact that the function~$g_\beta$ itself is monotone with respect to its first four parameters (see Equation~\eqref{eq:2nd_rho0x} and definitions thereafter).

We consider the usual partial order in~$\{-1, +1\}^{\Z^2}$ defined by vertex-wise comparison, that is, $\omega \prec \Omega \in \{-1, +1\}^{\Z^2}$ if and only if~$\omega(x) \leq \Omega(x)$ for every~$x \in \Z^2$. We say that an event~$A$ measurable with respect to the state of the vertices at time~${\t}$ is \emph{increasing} if
\begin{equation}
  \omega \prec \Omega \in \{-1, +1\}^{\Z^2}, \, \omega \in A \implies \Omega \in A.
\end{equation}
Fix now two events~$A, B$ that are increasing according to the above definition.
We then use Corollary~1.2 from~\cite{harris1977} and Proposition~9.3 from \cite{sullivan75} to conclude that
\begin{equation}
  \label{e:fkg1}
  \P_\rho \big( A \cap B \big)
  \geq
  \P_\rho \big( A \big)\P_\rho \big( B \big).
\end{equation}

Let~$\Lambda$ be a proper subset of~$\Z^2$. Given~$\omega \in \{-1, +1\}^{\partial_{ext} \Lambda}$, we define the process~$\sigma_{\t}|_{\Lambda, \omega}$ with boundary condition~$\omega$ in a similar manner to the original process: we let~$\{\sigma_0|_{\Lambda, \omega}(x); x \in \Lambda\}$ be i.i.d.\ random variables with distribution~$\rho \delta_{1} + (1 - \rho) \delta_{-1}$; consider Poisson clocks with intensity~$1$ for each site of~$\Lambda$, each clock tagged with a uniform random variable; and at each time in which a clock rings, we update the state of the process at the associated site using the function~$g_\beta$ as in~\eqref{e:gupdate}, the uniform random variable associated to the time, and the state of the process at the four neighbors, some of which might be in~$\partial_{ext} \Lambda$ and therefore have states given by~$\omega$. By the same graphical construction coupling, one shows that the probability that~$\sigma_{\t}|_{\Lambda, \omega}$ belongs to any increasing event is monotone in the boundary condition~$\omega$.

The following decoupling result is a direct consequence of Lemma~\ref{l:light}. It tells us that the state of~$\sigma_{\t}$ in two sufficiently distant sets is almost uncorrelated. Its straightforward proof is essentially the same as the one of Proposition~3.6 of~\cite{alves2019sharp}.

\nc{c:decouple}
\begin{lemma}
  \label{l:decouple}
  Fix two sets $B_1, B_2 \subseteq \mathbb{Z}^2$ (at least one of which is finite) and random variables $f_1, f_2: \{-1,+1\}^{\Z^{2}} \to [-1, 1]$, such that $f_i(\omega)$ is determined by the restriction of $\omega$ to the set $B_{i}$, for $i = 1, 2$.
  Then there exists a constant $\uc{c:decouple} > 0$ depending only on ${\t}$ and $\beta$ such that
  \begin{equation}
    \label{e:decouple}
    \begin{split}
      \E_\rho(f_1(\sigma_{\t}) f_2(\sigma_{\t})) \leq & \E_\rho(f_1(\sigma_{\t})) \E_\rho(f_2(\sigma_{\t}))\\
      & + \uc{c:decouple}^{-1} \min \{ |B_1|, |B_2| \}
      \exp \Big\{ - \uc{c:decouple} \dist(B_1, B_2) \log \big( \dist(B_1, B_2) \big) \Big\}.
    \end{split}
  \end{equation}
\end{lemma}

As usual in two-dimensional percolation, our results strongly rely on~\emph{crossing events} and~\emph{duality}.
We define the {\bf horizontal crossing event} of the box $[0, n] \times [0, m]$ by vertices with state~$+1$ at time~$\t$:
\begin{equation}
  \label{eq:horizon_cross_p}
  \mathcal{C}(n, m) :=
    \Big\{
      \{0\} \times [0, m] \xleftrightarrow{+,\,\, [0, n] \times [0, m]} \{n\} \times [0, m]
    \Big\} ,
\end{equation}
and we analogously define the $(-, *)$-crossing event
\begin{equation}
  \label{eq:horizon_cross_m}
  \mathcal{C}^*(n, m) :=
    \Big\{
      \{0\} \times [0, m] \xleftrightarrow{-,\,\, *,\,\, [0, n] \times [0, m]} \{n\} \times [0, m]
    \Big\} .
\end{equation}
By {\bf duality} we mean the fact that, relying on an elementary discrete topology consideration and invariance of the process under rotation by~$\pi/2$, the occurrence of the event~$\mathcal{C}(n, m)^c$ implies the occurrence of a translated and $\pi/2$-rotated version of the event~$\mathcal{C}^*(m, n)$, yielding
\begin{equation}
  \label{eq:duality_explanation}
  \P_\rho \big( \mathcal{C}(n, m)^c \big) = \P_\rho \big( \mathcal{C}^*(m, n) \big).
\end{equation}

One can prove the following result using the decoupling inequality provided by Lemma~\ref{l:decouple} together with a multiscale renormalization argument:
\nc{c:fsc}
\begin{lemma}[Finite-size criterion]
  \label{l:fsc}
  There exists a constant $\uc{c:fsc} = \uc{c:fsc}({\t}, \beta)$ such that, if
  \begin{display}\label{eq:fsc_condition}
    for some $n \geq 1$ we have $\P_\rho(\mathcal{C}(3n, n)) > 1 - \uc{c:fsc}$,
  \end{display}
  then $\rho > \rho_c$ and, for \emph{every} $n \geq 1$,
  \begin{equation}\label{eq:super_decay}
  \P_\rho(\mathcal{C}(3n, n))
  \geq 1 - \uc{c:fsc}^{-1} \exp \bigg\{ - \uc{c:fsc} \frac{n}{\big(\log(n)\big)^\epsilon} \bigg\}.
  \end{equation}
  An analogous result holds for dual crossings, that is, if
    \begin{display}
    for some $n \geq 1$ we have $\P_\rho(\mathcal{C}^*(3n, n)) > 1 - \uc{c:fsc}$,
  \end{display}
  then $\rho < \rho^*_c$ and, for every $n \geq 1$,
  \begin{equation}\label{eq:sub_decay}
  \P_\rho(\mathcal{C}^*(3n, n))
  \geq 1 - \uc{c:fsc}^{-1} \exp \bigg\{ - \uc{c:fsc} \frac{n}{\big(\log(n)\big)^\epsilon} \bigg\}.
  \end{equation}
\end{lemma}

The proof of this result follows from Proposition~$8.1$ of~\cite{alves2019sharp} and the duality of the process.
To understand this correspondence, it is important to observe that by the FKG inequality and the symmetries of the process,
\begin{equation}
  \label{e:anulus_rectangle}
  \P_\rho \big( \partial B(0, n) \xleftrightarrow{+} \partial B(0, 3n) \big) \leq 1 - \P_\rho(\mathcal{C}^*(3n, n))^4,
\end{equation}
so that $\P_\beta (\mathcal{C}^*(3n, n))^4$ being close to $1$ implies that a primal crossing of the annulus is improbable.
This low probability for some large~$n$ is one of the hypotheses of Proposition~$8.1$ of~\cite{alves2019sharp}, the other being a decoupling inequality also satisfied by the Glauber dynamics percolation (Lemma \ref{l:decouple}). We do not state the finite-size criterion for some sufficiently large $n$, but we can overcome this difference by taking $\uc{c:fsc}$ sufficiently small. Indeed, in order for the proof of Proposition~$8.1$ of~\cite{alves2019sharp} to go through, we need that there exists $n_0$ so that the probability of crossing an annulus of side $3n_0$ is smaller than a given $\varepsilon$ depending on the dimension. By choosing $\uc{c:fsc}$ sufficiently small, a tiling argument and the union bound imply that \ref{eq:fsc_condition} is enough to satisfy these requirements.

There are a few observations about the generality in which the lemma above applies that will be useful to us:

\begin{remark}\label{remark:dependence_constants}
The constant $\uc{c:fsc}$ in the lemma above depends only on the speed of the correlation decay of the model. This will be useful when proving Theorem~\ref{t:transition}.
Also, the Lemma still remains valid if one has access to weaker decorrelation estimates: if the decay is only exponential, the same proof still applies, but one obtains stretched-exponential bounds in~\eqref{eq:super_decay} and~\eqref{eq:sub_decay}.
\end{remark}

We can now define the~\emph{fictitious regime} of the parameter~$\rho$, where there is a uniformly positive probability in~$n$ of both $+$-crossings and $(-, *)$-crossings of rectangles~$[0, 3n] \times [0, n]$ in the hard direction. We define
\begin{equation}\label{eq:rho_+}
  \begin{split}
    \rho_+ & = \inf \Big\{ \rho \in [0,1]; \P_\rho (\mathcal{C}(3n, n)) > 1 - \uc{c:fsc} \text{ for some $n \geq 1$} \Big\}\\
    & = \sup \Big\{ \rho \in [0,1]; \P_\rho (\mathcal{C}^*(n, 3n)) \geq \uc{c:fsc} \text{ for all $n \geq 1$} \Big\};
  \end{split}
\end{equation}
\begin{equation}\label{eq:rho_-}
  \begin{split}
    \rho_- & = \sup \Big\{ \rho \in [0,1]; \P_\rho (\mathcal{C}^*(3n, n)) > 1 - \uc{c:fsc} \text{ for some $n \geq 1$} \Big\}\\
    & = \inf \Big\{ \rho \in [0,1]; \P_\rho (\mathcal{C}(n, 3n)) \geq \uc{c:fsc} \text{ for all $n \geq 1$} \Big\}.
  \end{split}
\end{equation}
It is clear by monotonicity and duality of the process that $\rho_- \leq \rho_+$.  Our whole objective is to show~\emph{sharpness}, that is, that $\rho_- = \rho_+$ and the interval~$[\rho_-, \rho_+]$ is degenerate.

One of the the main tools in the study of planar percolation has been the Russo-Seymor-Welsh theory (RSW for short). Developed originally for Bernoulli percolation, this theory essentially says that, if the probability of crossing rectangles with fixed aspect-ratio but variable size in the ``easy direction'' is uniformly positive in said size, then the probability of crossing in the hard direction is also uniformly positive. In the past decade there was a substantial development aiming at generalizing the theory for other planar models. In~\cite{tassion2016crossing}, this theory was developed for Voronoi percolation; in \cite{ATT18}, the results were generalized for non-i.i.d.\ processes. The latest development in this generalization effort proves the result in the greatest generality yet: for planar models satisfying the Harris-FKG inequality and invariance under translations and the~$D_4$ symmetries. These results can all be adapted to our context. In particular, the hypotheses of~\cite{tassion-koeler20} are all satisfied by our process, implying the existence of a constant~$\uc{c:rsw} > 0$ possibly depending on~$\beta, \t$ such that
\nc{c:rsw}
\begin{display}
  \label{eq:rsw}
  for all $\rho \in [\rho_-, \rho_+]$ we have\\
  $\P_\rho \big( \mathcal{C}(3n, n) \big) > \uc{c:rsw}$ and
  $\P_\rho \big( \mathcal{C}^*(3n, n) \big) > \uc{c:rsw}$, for all $n \geq 1$.
\end{display}

To finish this section, we define the~\emph{arm event} we will need later. For~$m < n \in \N$, we write
\begin{equation}
  \begin{split}
    \Arm(m, n) &= \big\{ \partial B(0, m) \xleftrightarrow{+} \partial_{ext} B(0, n) \big\},
    \\
    \Arm_x(m, n) &= \big\{ \partial B(x, m) \xleftrightarrow{+} \partial_{ext} B(x, n) \big\}.
  \end{split}
\end{equation}

\section{Existence of phase transitions}
\label{s:phase_transition_of_glauber}
~
\par We here prove Theorem~\ref{t:transition}. We will first treat the case of existence of phase transition for $\beta=\infty$ or for small times. These rely heavily on Lemma~\ref{l:fsc} and continuity arguments. The proof of absence of phase transition for small values of $\beta$ and large values of time is more delicate and relies on the fact that the same is true for large-temperature Ising model.

\begin{remark}
Exclusively in this section, we denote by $\P_{\rho}^{\beta}$ the distribution of the process with initial density $\rho$ and the inverse temperature $\beta \in [0, \infty]$. For $\beta < \beta_c$, we let $\P^\beta$ denote the unique infinite-volume measure of the Ising model.
\end{remark}

\bigskip

\noindent\textbf{The case $\beta= \infty$.} This is the simplest case. We first fix $\tau \geq 0$ and apply Lemma~\ref{l:fsc}, which implies the existence of a constant $\uc{c:fsc} = \uc{c:fsc}(\t, \beta) > 0$ such that, if
\begin{display}
  $\P_\rho^{\infty}(\sigma_\t \in \mathcal{C}(3n, n)) > 1 - \uc{c:fsc}$, for some $n \geq 1$,
\end{display}
then percolation occurs. We now conclude by noticing that $\rho \mapsto \P^\infty_\rho(\sigma_\t \in \mathcal{C}(3n, n))$ is continuous and that this probability equals $1$ if $\rho=1$, implying $\rho_c(\t) < 1$. Arguing analogously for the event $\mathcal{C}^*(3n, n)$, we obtain $\rho_c(\t) \in (0, 1)$.

\bigskip

\noindent\textbf{Non-trivial phase transition for small times.} Assume now that $\beta \in (0, \infty)$. Our goal is to prove that there exists a non-trivial phase transition for small values of $\t$. In this case, we make use of Remark~\ref{remark:dependence_constants}, since the correlation decay in Lemma~\ref{l:light} can be made uniform for compact sets of time. Let us now restrict ourselves to $\t \in [0,1]$. In this case, the constant $\uc{c:fsc}$ from Lemma~\ref{l:fsc} is uniformly bounded. Furthermore, the function $\t \mapsto \P^\beta_\rho(\sigma_{\t} \in \mathcal{C}(3n, n))$ is continuous.
Observe that $\t = 0$ corresponds to independent site-percolation on the plane, which undergoes a non-trivial (and sharp) phase transition.
Therefore we conclude the existence of a non-trivial phase transition for small values of $\t$ as well.

\bigskip

\noindent\textbf{Absence of phase transition.} This is the case that demands a more careful analysis. Our first goal is to prove a uniform decoupling inequality for small values of $\beta$. This will then allow us to employ Lemma~\ref{l:fsc} together with arguments that compare finite-time configurations with the limiting distribution.

We first describe an alternative exploration process that determines the value of  $\sigma_{\t}(0)$. Fix a realization of the graphical construction of the process $(\sigma_{\t})_{\t \geq 0}$, and notice that, according to~\eqref{eq:update_definition}, if $T^{x}_{i}$ is a mark at $x \in \Z^{2}$, then the new opinion of $x$ does not depend on the values of its neighboring vertices if $R_{i}^{x} \leq S_{\beta}(-1,-1,-1,-1)$ or $R_{i}^{x} > S_{\beta}(1,1,1,1)$.

We now describe a random continuous graph that can be used to determine $\sigma_{\t}(0)$. We start by declaring the space-time point $(0,\t)$ as active, and go backwards in time  from $(0,\t)$ until we hit a Poisson mark, say $T^{0}_{i}$. The mark $T^{0}_{i}$ is called a death mark of the process if the opinion of the vertex at that time can be determined without the information of the neighbors. In this case, we also declare the space-time vertex $(0,T^{0}_{i})$ as a dead end. If this is not the case, the vertex is declared inactive while all neighboring vertices are declared active. We now continue the process until we reach time $0$ or until no more active vertices remain. Observe that, given the exploration graph and the opinion of all dead ends of the graph (which are determined by the value of the corresponding uniform random variable), the value of $\sigma_{\t}(0)$ is determined.

We now verify that this exploration algorithm implies that the exponential correlation decay is uniform for all large enough $\t$, if $\beta$ is taken small enough. This will be achieved by comparing the exploration process with a Galton-Watson random tree. Notice that the probability that a Poisson mark in the exploration is not a dead end is $S_{\beta}(1,1,1,1)-S_{\beta}(-1,-1,-1,-1)$ and, in this case, it gives birth to four new vertices. This implies that the distance (in $\Z^{2}$) one needs to observe is bounded by the number of generations on a Galton-Watson random tree with offspring distribution $\big(p_{\beta}(k) \big)_{k \geq 0}$ given by
\begin{equation}\label{eq:offspring}
p_{\beta}(4) = 1-p_{\beta}(0) = S_{\beta}(1,1,1,1)-S_{\beta}(-1,-1,-1,-1).
\end{equation}

\nc{c:correlation_small_beta}

In particular, we obtain the following lemma.
\begin{lemma}\label{lemma:correlation_decay_small_beta}
There exists $\beta_{0}>0$ such that, for all $\beta \leq \beta_{0}$, there exists $\t_{0}=\t_{0}(\beta)$ and a positive constant $\uc{c:correlation_small_beta}=\uc{c:correlation_small_beta}(\beta)$ such that, for all $\t \geq \t_{0}$ and $\rho \in [0,1]$,
\begin{equation}
\Cov^{\beta}_{\rho}(\sigma_{\t}(x), \sigma_{\t}(y)) \leq \uc{c:correlation_small_beta} e^{-\uc{c:correlation_small_beta}^{-1}|x-y|_{1}}.
\end{equation}
\end{lemma}

\begin{proof}
Notice that both $S_{\beta}(1,1,1,1)$ and $S_{\beta}(-1,-1,-1,-1)$ converge to $\frac{1}{2}$ as $\beta \to 0$. In particular, there exists $\beta_{0}$ such that $p_{\beta}(4) < \frac{1}{4}$, for all $\beta \leq \beta_{0}$. This then implies that the Galton-Watson tree with offspring distribution $p_{\beta}$ is subcritical. Let $Z_{n}$ denote the number of descendants at generation $n$ of this tree.

The distance one needs to observe in the exploration algorithm in order to determine the opinion of a vertex $x$ is bounded from above by the maximal nonempty generation of a Galton-Watson tree with offspring distribution $p_{\beta}$. Besides, if two vertices have disjoint exploration graphs, their opinions at time $\t$ are independent. In particular, we obtain that
\begin{equation}
\begin{split}
\Cov^{\beta}_{\rho}(\sigma_{\t}(x), \sigma_{\t}(y)) & \leq 2 \P_{\rho}^{\beta}\Big[ \inf\{n: Z_{n} >0\} \geq \frac{1}{2}|x-y|_{1}\Big] \\
& \leq  2\P_{\rho}^{\beta}\big[ Z_{\frac{1}{2}|x-y|_{1}} \geq 1 \big] \leq 2(4p_{\beta}(4))^{\frac{1}{2}|x-y|_{1}},
\end{split}
\end{equation}
concluding the proof.
\end{proof}

Let us now investigate percolation on the Ising model. In~\cite{higuchi1}, the author proves that the Ising model undergoes a sharp phase transition on the external field for every value of $\beta$. With the results in~\cite{higuchi2}, we obtain that the case of zero external field is subcritical for $\beta$ small enough and \cite{higuchi2} also implies
\begin{equation}
\P^{\beta}[\mathcal{C}(3n,n)] \to 0,
\end{equation}
for the Ising model with inverse temperature $\beta$ small enough. In particular, we have
\begin{equation}\label{eq:subcritical_ising}
\P^{\beta}[\mathcal{C}^*(3n,n)] \to 1.
\end{equation}

Our last ingredient is the weak convergence of the Glauber dynamics to the Ising model (see for example~\cite{martinelli_olivieri}). We have that
\begin{equation}
\P_{\rho}^{\beta}[\sigma_{\t} \in \mathcal{C}^*(3n,n)] \to \P^{\beta}[\mathcal{C}^*(3n,n)], \text{ as } \t \to \infty,
\end{equation}
uniformly in $\rho \in [0,1]$.

The proof will now be completed with the aid of Lemma~\ref{l:fsc}. Let $\beta_{0}$ as in Lemma~\ref{lemma:correlation_decay_small_beta}, and fix $\beta \leq \beta_{0}$. Once again, thanks to Remark~\ref{remark:dependence_constants}, the constant $\uc{c:fsc}$ in Lemma~\ref{l:fsc} does not depend on $\t \geq \t_{0}(\beta)$. Take then $n$ large enough such that
\begin{equation}
\P^{\beta}[\mathcal{C}^*(3n,n)] \geq 1-\frac{\uc{c:fsc}}{2},
\end{equation}
and $\t$ large so that
\begin{equation}
\sup_{\rho \in [0,1]} \big| \P_{\rho}^{\beta}[\sigma_{\t} \in \mathcal{C}^*(3n,n)] - \P^{\beta}[\mathcal{C}^*(3n,n)] \big| \leq \frac{\uc{c:fsc}}{2}.
\end{equation}

In particular, this implies
\begin{equation}
\P_{\rho}^{\beta}[\sigma_{\t} \in \mathcal{C}^*(3n,n)] \geq 1-\uc{c:fsc},
\end{equation}
and hence there is no percolation for all values of $\rho$ by Remark \ref{remark:dependence_constants} and Lemma~\ref{l:fsc}.

\section{Sharpness using OSSS and comparison of pivotal probabilities}
\label{s:proof_idea}
~
\par The last decade has been marked by a frequent use of the OSSS inequality \cite{OSSS} to prove sharpness results in statistical mechanics, see \cite{D-CRT19, D-CRT19b, duminil2020subcritical, D-CGRS20}.

The success of this technique suggests the existence of a more general result or procedure that would apply to most models of interest.
This is however not the case, as even proofs that make use of OSSS still require model specific arguments, as we point out below.

Before we enter into details of the flexibility (and also the limitations) of the OSSS techniques, let us first state the main theorem.

\paragraph{The OSSS inequality.}

Given a natural number $n \geq 1$ and a Boolean function $f: \{-1, 1\}^n \to \{-1, 1\}$, we now define what we mean by a \emph{randomized algorithm to determine $f$}.
Although an algorithm is usually described in terms of decision trees, we give here a precise definition using notation from stochastic processes.
For this, fix $f$ as above and a (possibly random) argument $\omega \in \{-1, 1\}^n$ on which we want to evaluate $f$.

The first ingredient of an algorithm is a process $(X_1, X_2, \dots, X_n)$, where each $X_j \in \{1, \dots, n\}$ and $X_{i} \neq X_{j}$ if $i \neq j$ (this can be seen as a random permutation of $\{1, \dots, n\}$, giving us the order in which we reveal $\omega$).
The second ingredient in the construction is a random time $T \in \{1, \dots, n\}$ at which we stop the revealment process.

Introducing the filtration $\mathcal{F}_j = \sigma \big( X_1, \omega(X_1), X_2, \omega(X_2), \dots, X_j, \omega(X_j) \big)$, we say that $\mathcal{A} = \big( (X_j)_{j=1}^n, T \big)$ is a \emph{random algorithm to determine $f$} if:
\begin{itemize}
\item $\mathcal{F}_0$ is independent of $\omega$,
\item $(X_j)_{j = 1}^n$ is predictable (meaning that $X_j \in \mathcal{F}_{j - 1}$ for every $j = 1, \dots, n$),
\item $T$ is a stopping time with respect to the filtration $(\mathcal{F}_j)_{j = 0}^n$, and
\item $f$ is measurable with respect to $\mathcal{F}_T$.
\end{itemize}

Having this definition, we can state the OSSS inequality.

\begin{theorem}[\cite{OSSS}]
  \label{t:OSSS}
  Let $f: \{-1, 1\}^n \to \{-1, 1\}$ and let $\mathcal{A}$ be a random algorithm to determine $f$.
  Then, considering a uniform measure on the hypercube $\{-1, 1\}^n$,
  \begin{equation}
    \label{e:OSSS}
    \Var(f) \leq \sum_{i = 1}^n \delta_i(\mathcal{A}) \Inf_i(f),
  \end{equation}
  where $\delta_i$ is the probability that $\mathcal{A}$ reveals $i$ and $\Inf_i$ is the influence of $i$, that is
  \begin{equation}
    \label{e:rev}
    \delta_i(\mathcal{A}) = \P [X_j = i; \text{ for some $j \leq T$}],
  \end{equation}
  and
  \begin{equation}
    \label{e:inf}
    \Inf_i(f) = \P [f(\omega_1, \dots, \omega_i, \dots, \omega_n) \neq f(\omega_1, \dots, -\omega_i, \dots, \omega_n)].
  \end{equation}
\end{theorem}

We now move to a simple yet powerful application of OSSS to independent percolation.

\paragraph{Sharpness for Bernoulli percolation.}

Let us now give a quick overview of how OSSS can be used to prove sharpness for Bernoulli percolation.
Although this is a very classical result and technique, we have decided to include this brief overview, since our proofs are inspired by this argument.

\vspace{2mm}

We start by choosing a parameter $p \in [0, 1]$ and letting $(\sigma(x))_{x \in \mathbb{Z}^2}$ be an i.i.d.\ collection of $\Ber(p)$ random variables under $\P_p$.

Recall the definition of the crossing probabilities $\mathcal{C}(n, m)$ in \eqref{eq:horizon_cross_p}.
What OSSS will allow us to do is to prove that the functions $h_n: [0, 1] \to [0, 1]$ defined through the map $p \mapsto \P_p[\mathcal{C}(n, n)]$ undergo a sharp transition from zero to one, as $n$ goes to infinity.

\nc{c:box_cross}
The way in which this is done is by bounding from below the derivative of $h_n$ in the critical window $\big( p_-(n), p_+(n) \big)$ where
\begin{equation}
  \label{e:p_-+}
  \begin{split}
    p_-(n) & := \sup \big\{ p \in [0, 1]; \P_p[\mathcal{C}(n, n) < \uc{c:box_cross}] \big\} \text{ and}\\
    p_+(n) & := \inf \big\{ p \in [0, 1]; \P_p[\mathcal{C}(n, n) > 1-\uc{c:box_cross}] \big\},
  \end{split}
\end{equation}
for some $\uc{c:box_cross}>0$. It is a consequence of Russo Seymour Welsh's inequality and a renormalization argument by Kesten (see also Lemma~\ref{l:fsc}) that: if $p \leq p_-(n)$ for some $n \geq 1$, then $\P_p[0 \leftrightarrow \partial B(0, n)]$ decays exponentially fast.
This is what is called a finite size criterion.
Analogously, if $p \geq p_+(n)$ for some $n \geq 1$, the same exponential decay is observed for dual paths.
Therefore, the statement of sharpness of the model can be reduced to proving that $|p_+(n) - p_-(n)|$ converges to zero with $n$.

\nc{c:bernoulli_arm}
We will also need the following bounds, which are consequences of symmetry and the RSW's inequality.
There exists $\uc{c:bernoulli_arm} > 0$ such that for every $n \geq 1$,
\begin{equation}
  \label{e:arm_cross}
  \P_p [\Arm(1, n)] \leq n^{-\uc{c:bernoulli_arm}}
  \qquad \text{and} \qquad
  \P_p [\mathcal{C}(n, n)] \in (\uc{c:box_cross}, 1 - \uc{c:box_cross}),
\end{equation}
for every $p \in \big( p_-(n), p_+(n) \big)$.

Although this is not the central point of this article, we have decided to introduce in more detail the algorithm that is used, since it will serve as a basis for the other algorithms that will come.

Our algorithm starts by selecting a random integer $Z$ uniformly on $\{1, \dots, n\}$.
We then explore all the sites that are connected by open sites to some neighbor of the column $\{Z\} \times \{1, \dots, n\}$.

Here are few observations from the above construction:
\begin{itemize}
\item All the sites in the column $\{Z\} \times \{1, \dots, n\}$ will be revealed,
\item if a site is revealed, it must be connected to the above column,
\item by the time we have explored all these sites, we are able to decide if $\mathcal{C}(n, n)$ occurred.
\end{itemize}
Or in other words, this algorithm determines the crossing events.

We can now apply the OSSS inequality \eqref{e:OSSS} to obtain that for every $p \in (0, 1)$,
\begin{equation}
  \label{e:upper_variance}
  \begin{split}
    \Var_p(\mathcal{C}(n, n)) & \leq \sum_{i = 1}^{n^2} \delta_i(\mathcal{A}) \Inf_i(\mathcal{C}(n, n))\\
    & \leq \sup_{i \in \{1, \dots, n\}^2} \P_p \big[ i \leftrightarrow \{Z\} \times \{1, \dots, n\} \big] \sum_{i = 1}^{n^2} \Inf_i(\mathcal{C}(n, n)).
  \end{split}
\end{equation}
We estimate, for every $(i_1, i_2) \in \{1, \dots, n\}^2$ and $p \in \big( p_-(n), p_+(n) \big)$
\nc{c:column}
\begin{equation}
  \label{e:touch_column}
  \begin{split}
    \P_p \big[ (i_1, i_2) & \leftrightarrow \{Z\} \times \{1, \dots, n\} \big] \\
    & \leq \P_p \big[ | i_1 - Z | \leq \sqrt{n} \big]
    + \P_p \big[ \Arm(1, \sqrt{n}) \big]
    \overset{\eqref{e:arm_cross}}\leq \frac{\sqrt{n}}{n} + n^{-\frac{\uc{c:bernoulli_arm}}{2}} \leq 2 n^{-\uc{c:column}}.
  \end{split}
\end{equation}
Finally, we recall Russo's formula, to estimate for $p \in \big( p_-(n), p_+(n) \big)$
\begin{equation}
  \label{e:high_derivative}
  \begin{split}
    \frac{\d \P_p [\mathcal{C}(n, n)]}{\d p} (p)
    & = \sum_{i = 1}^{n^2} \Inf_i \big( \mathcal{C}(n, n) \big)
    \overset{\eqref{e:upper_variance}, \eqref{e:touch_column}}\geq
    \frac{n^{\uc{c:column}}}{2} \Var_{p} (\mathcal{C}(n, n)) \\
    & = \frac{n^{\uc{c:column}}}{2} \P_{p} [\mathcal{C}(n, n)] \big( 1 - \P_{p} [\mathcal{C}(n, n)] \big)
    \overset{\eqref{e:arm_cross}}\geq \frac{\uc{c:box_cross}^2}{2} n^{\uc{c:column}},
  \end{split}
\end{equation}
which is the key inequality that is used to prove the sharpness of the phase transition of the model.
Indeed, from \eqref{e:high_derivative} and the fact that $\P_p[\mathcal{C}(n, n)] \in [0, 1]$ we conclude that
\begin{equation}
  p_+(n) - p_-(n) \leq \frac{\uc{c:box_cross}^2}{2n^{\uc{c:column}}},
\end{equation}
which converges to zero with $n$, proving that the phase transition is sharp.

\begin{remark}
  \label{r:coincidence}
  It is important to observe here that a fortunate coincidence has helped us in the last step of the proof.
  More precisely, note that the sum of influences that appear on the right hand side of \eqref{e:upper_variance} was originated from the OSSS inequality, but incidentally this happens to be the same sum that appears in Russo's formula.
  This is not the case for our models as it is further commented in Remark~\ref{r:difficulties}.2 below.
\end{remark}

When we apply this technique to our model, it will become clear how this can be used to prove sharpness of the phase transition, as well as a polynomial upper bound on the size of the critical window.

\paragraph{Difficulties applying the OSSS technique to other models.}

We now discuss the possible strategies and difficulties involved in generalizing the method described above to non-independent percolation models.

The first obvious obstruction comes from the lack of independence, since the original OSSS inequality \eqref{e:OSSS} is stated under this assumption.
It is a common technique however to use independent Bernoulli random variables to simulate a dependent environment, as done in \cite{D-CRT19}.

In fact we have already introduced independent random variables (namely, $\sigma_0(x)$ and $R^x_i$) in the graphical construction.
However there are two difficulties that need to be addressed as we comment in the following.

\begin{remark}
  \label{r:difficulties}
  In order to apply the above techniques to non-independent scenarios, we have to take into account the following observations.
  \begin{enumerate}
  \item When devising a graphical construction, it is rarely the case that all sources of randomness are independent Bernoulli random variables, for example in our case the marks $T^x_i$ are distributed as a Poisson process.
    Our strategy to deal with this issue is to condition on all the elements of the graphical construction that are not independent Bernoulli variables.
    The challenge then is to show that these variables that we conditioned on are not very influential.

    A very important contribution in this direction comes from the work of \cite{ABGM14}, and we discuss this technique in detail in Lemma~\ref{l:thinning}.
  \item Suppose for simplicity that our model is constructed from independent Bernoulli distributed random variables only.
    Even in that case there can be a difficulty in applying OSSS that comes from the different roles that these variables can play.
    Taking our Glauber dynamics as an example: there are variables used only for the dynamics itself, while the variables $\sigma_0(x)$ give the initial condition.

    To understand better why this poses a problem, let us recall that each of these variables appear in the OSSS formula, since they can all influence our event.
    However, only some of them appear in Russo's formula (in our Glauber example, only the $\sigma_0(x)$).
    This breaks the fortunate coincidence that occurs in the Bernoulli case and we commented in Remark~\ref{r:coincidence}.

    The solution to this second issue is inspired by the work of Aizenman-Grimmett \cite{AG91}.
    Roughly speaking, the strategy here is to compare the probability that different types of variables are pivotal, making it possible to connect OSSS and Russo again.
    This is explained in detail in Section~\ref{s:AG}.
  \end{enumerate}
\end{remark}

\paragraph{Homogenization by thickening.}

In order to prove a result of the form of Theorem~\ref{t:phases}, an intuitive initial idea would be to condition on the realization of the Poisson marks and carry out the approach described in this section considering the quenched measure. This is not easily feasible, since it might be the case that the realization of the Poisson is not well behaved and does not allow one to extrapolate bounds of annealed probabilities to the quenched case. To circumvent this difficulty, one needs to work with quenched probabilities for which information on conditional crossing probabilities can still be obtained. It turns out this is possible if one considers a ``thicker'' graphical construction, by fixing a collection of Poisson marks with larger density and adding extra randomness that sometimes ignores the extra clock rings. This technique was developed in \cite{ABGM14} and allows one to show that conditioning on the location of the thicker Poisson marks of our graphical construction cannot have a big influence on the crossing probabilities.

The central lemma is in fact much more general and applies to several contexts in which one wants to condition on a Poisson point process, while still controlling the probability of certain events.

Fix an integer $k \geq 1$ and let $\mu = \sum_{i} \delta_{z_i}$ be a Poisson point process on a Polish space $M$ with intensity measure $k \nu$.

We define a thinning $\mu_{k}$ of $\mu$, which is obtained by keeping each point of $\mu$ independently with probability $1 / k$.
Clearly $\mu_{k}$ is a Poisson point process with intensity $\nu$.

The following lemma, that first appeared in~\cite{ABGM14} for the particular case of the Gilbert disc model, is central in our analysis.	
\begin{lemma}[\cite{ABGM14}]
  \label{l:thinning}
  In the above context, let $Y$ be a $\{0, 1\}$-valued random variable, measurable with respect to $\mu_{k}$, and
  \begin{equation}
    Z = \E \big[ Y(\mu_{k}) | \mu \big].
  \end{equation}
  In this context,
  \begin{equation}
    \label{e:bound_var_up}
    \Var(Z) \leq 1 / k.
  \end{equation}
  Moreover, if $\E \big( Y \big) \in (c, 1-c)$, for some $c \in [0, \frac{1}{2}]$, then
  \begin{equation}
    \label{e:bound_var_down}
    \P \Big[ \E \big( Y(\mu_{k}) | \mu \big) \notin \Big(\frac{c}{2}, 1 - \frac{c}{2}\Big) \Big] \leq \frac{4\Var(Z)}{c^2} \leq \frac{4}{k c^2}.
  \end{equation}
\end{lemma}

\begin{proof}
  Let $\tilde{\mu} = \sum_i \delta_{(z_i, d_i)}$ be a Poisson point process on a space $M \times [k]$, with intensity measure $\nu \otimes \Count_{[k]}$, where $\Count_{[k]}$ is the counting measure in $[k]=\{1, \dots, k\}$.
  We write $\tilde{\mu}_M = \sum_i \delta_{z_i}$ for the projection of $\tilde{\mu}$ on the first coordinate, which is also a Poisson point process with intensity $k\nu$. For $j \in [k]$, let $\tilde{\mu}_j = \sum_{i: d_i =j} \delta_{z_i}$. Consider also an independent random variable $X$ uniformly distributed in $[k]$ and observe that the pair $(\tilde{\mu}_{X}, \tilde{\mu}_{M})$ has the same distribution as $\tilde{\mu}_{k}, \tilde{\mu})$. From this we compute
  \begin{equation}
    \begin{split}
      \Var(Z) & = \Var \big( \E \big[ Y(\tilde{\mu}_{X}) | \tilde{\mu}_{M} \big] \big)
      \leq \Var \big( \E \big[ Y(\tilde{\mu}_{X}) | (\tilde{\mu}_{j})_{j=1}^{k} \big] \big) \\
      & \ = \Var \Big( \frac{1}{k} \sum_{j=1}^{k}Y(\tilde{\mu}_{j}) \Big)
      \leq \frac{1}{k^{2}} \sum_{j=1}^{k}\Var \big(Y(\tilde{\mu}_{j}) \big)
      \leq \frac{1}{k},
    \end{split}
  \end{equation}
concluding the proof.
\end{proof}

\section{Proof ingredients}\label{s:proof_ingredients}

\paragraph{The thickened graphical construction.}In this section we will define a process from which we can extract the percolation process defined in Section~\ref{s:model_and_notations}. In line with Lemma~\ref{l:thinning} and the discussion from the previous section, we will first thicken the Poisson mark process by a factor of~$k \geq 1$, and keep each mark with probability~$1 / k$. We will also couple all possible starting densities~$\rho$ using a standard uniform coupling.

We first associate to each~$x \in \Z^2$ an independent random variable~$U^x \eqd U[0, 1]$. We can then define
\begin{equation}
  \label{eq:2nd_rho0x}
  \sigma_{\rho, 0}(x) = {\bf 1}_{\{U^x < \rho\}} - {\bf 1}_{\{U^x \geq \rho\}},
\end{equation}
so that this collection of variables becomes coupled for different~$\rho \in [0, 1]$, all of them having distribution~$\Ber(\rho)$.

Let~$k \geq 1$ be some fixed integer. We define an independent Poisson point process~$\mu$ on $\mathbb{Z}^2 \times \mathbb{R}_+ \times [0, 1] \times \{0, 1\}$ with underlying intensity measure given by
\begin{equation}
  \Count_{\Z^2} \, \otimes \, k \cdot \d x \, \otimes \, U[0, 1] \, \otimes \, \Ber(1/k),
\end{equation}
where~$\Count_{\Z^2}$ is the counting measure over~$\Z^2$, and $k \cdot \d x$ represents $k$ times the Lebesgue measure over~$\R_+$. This point process will encode all the relevant information needed in order to define the process in the manner described above. A point~$(X, T, R, D)$ in the support of the point measure will carry the following information:

\begin{itemize}
  \item $X \in \mathbb{Z}^2$ will represent the site;

  \item $T \in \mathbb{R}_+$ will encode the time mark associated to~$X$;

  \item $R \in [0, 1]$ will hold the uniform random variable used in order to apply the rules of the local updates, that is, this random variable will be used in the same manner as~$R_x$ in~\eqref{e:gupdate};

  \item $D \in \{0, 1\}$ will be the Bernoulli random variable responsible for either accepting or rejecting the update of the process given by~$X$, $T$, and $R$. This is the thinning random variable appearing in Lemma~\ref{l:thinning}.
\end{itemize}

We fix an arbitrary manner of enumerating points in the support of the point measure~$\mu$. In this way, we can write
\begin{equation}
  \label{eq:mu_def}
  \mu := \sum_{i \geq 1} \delta_{(X_i, T_i, R_i, D_i)},
\end{equation}
so that~$X_i, T_i, R_i, D_i$ are well-defined random variables. We can then define~$\mu_k$, the \emph{thinned} version of~$\mu$, by first removing all points in the support such that~$D_i = 0$ and then projecting the remaining points onto~$\mathbb{Z}^2 \times \mathbb{R}_+ \times [0, 1]$. Given~$x \in \Z^2$, we consider the set of time marks of~$\mu_k$ associated to~$x$,
\begin{equation}
  \big\{
    T_i \in \R_+; \,
     (x, T_i, R_i) \in \supp(\mu_k)
  \big\}.
\end{equation}
These sets are distributed as the sets~$\mathcal{P}^x$ defined in Section~$\ref{s:model_and_notations}$, and the same notation will be used for them. The variables~$R_i$ associated to each~$T_i \in \mathcal{P}^x$ are then distributed as the collection of independent uniform random variables associated to each time mark of~$x$ defined in \eqref{eq:time_mark_uniform_rv}. We can then use these collections of Poissonian time marks, associated uniform random variables, and the initial state variables defined in~\eqref{eq:2nd_rho0x} in order to define the graphical construction in the same manner of Section~\ref{s:model_and_notations}. By elementary properties of the Poisson point process, the process defined in this manner has the same distribution as
\begin{equation*}
  \big( \sigma_{\rho, t}(x) \big)_{x \in \Z^2, t \in \R_+}.
\end{equation*}
From now on we assume the new construction above was the one used in order to define this process, and write $\P$ and $\E$ for the probability and expectation corresponding to the new construction. We note that all starting densities~$\rho \in [0, 1]$ are coupled, and the monotone properties alluded to in Section~$\ref{s:model_and_notations}$ are justified by the monotonicity of~$g_\beta$.

\paragraph{Influence and Russo's formula.} Consider a finite set~$K \subset \Z^2$ and an event $A$ measurable with respect to the variables in the collection $(\sigma_{\rho, \t}(x))_{x \in K}$. This collection depends on both the initial configuration $(U^x)_{x \in \mathbb{Z}^2}$ and the graphical construction $\mu_k$, constructed from $\mu = \sum_{i \geq 0} \delta_{(X_i, T_i, R_i, D_i)}$. In order to measure how sensitive the event~$A$ is to perturbations of these random variables, we change the configuration $\sigma_{\rho, \t}$ in different ways:
\begin{itemize}
\item \emph{Change initial distribution}: given $u \in [0, 1]$ and~$x \in \Z^2$, let $\sigma^{x, u}_{\rho, \t}$ be obtained in the same way as $\sigma_{\rho, \t}$, but replacing $U^x$ by $u$.
\item \emph{Change the thinning variable}: given $d \in \{0, 1\}$ and~$i \in \N$, let $\sigma^{i, d}_{\rho, \t}$ be obtained in the same way as $\sigma_{\rho, \t}$, but replacing $D_i$ in the expression of $\mu$ by the value~$d$.
\end{itemize}
This gives rise to different types of \emph{pivotality}:
\begin{itemize}
\item \emph{Initial distribution}: Given $(U^x)_{x \in \mathbb{Z}^2}$ and $\mu_k$, we say that $x$ is $\rho$-pivotal for $A$ if the occurrence of the event $A$ changes between the configurations $\sigma^{x, 0}_{\rho, \t}$ and~$\sigma^{x, 1}_{\rho, \t}$.
\item \emph{Thinning variables}: Given $(U^x)_{x \in \mathbb{Z}^2}$ and $\mu_k$, we say that $i$ is $\rho$-pivotal for $A$ if the occurrence of the event $A$ changes between the configurations $\sigma^{i, 0}_{\rho, \t}$ and~$\sigma^{i, 1}_{\rho, \t}$.
\end{itemize}

\begin{remark}
  Note that both definitions above have the same notation ($\rho$-pivotal).
  But one refers to some $x \in \mathbb{Z}^2$ and the other to some $i \in \N$.
\end{remark}

Given $\mu = \sum_{i \geq 0} \delta_{(X_i, T_i, R_i, D_i)}$, we define $\mu' := \sum_{i \geq 0} \delta_{(X_i, T_i, R_i)}$, that is, we forget the knowledge about which time marks are erased and which are kept. We will write $\P [ \, \cdot \, | \mu' ]$ for the conditional probability with respect to~$\mu'$. This means that the randomness in~$\P [ \, \cdot \, | \mu' ]$ comes from the uniform variables $(U^x)_{x \in \mathbb{Z}^2}$, which encode the initial state of the vertices, and the variables~$D_i$, which determine the thinning of variables related to the graphical construction.

Suppose $A$ is a non-decreasing event. We define
\begin{itemize}
\item \emph{Initial condition influence}: Given $x \in \mathbb{Z}^2$, let
  \begin{equation}
  \Inf^{\textnormal{init}}_\rho (x, A | \mu') = \P [x \text{ is $\rho$-pivotal for $A$} | \mu'].
\end{equation}
\item \emph{Thinning influence}: Given $i \in \N$, let
  \begin{equation}
  \Inf^{\textnormal{thin}}_\rho (i, A | \mu') = \P [i \text{ is $\rho$-pivotal for $A$} | \mu'].
\end{equation}
\end{itemize}

We can finally state Russo's formula applied to this process, relating the~$\rho$-derivative of the probability of an increasing event, and the initial condition influence:
\begin{lemma}
  \label{l:russo}
  Fix $\t \geq 0$ and let $A$ be a non-decreasing event that depends on the configuration $(\sigma_{\rho, \t}(x))_{x \in K}$, for a finite set $K \subseteq \mathbb{Z}^2$. Then
  \begin{equation}
    \label{e:russo}
    \frac{\d \P \big( \sigma_{\rho, \t} \in A \big| \mu' \big)}{\d \rho} = \sum_{x \in \mathbb{Z}^2} \Inf^{\textnormal{init}}_\rho \big( x, A | \mu' \big),
  \end{equation}
  for $\P$-almost every $\mu'$.
\end{lemma}

\begin{proof}[Proof sketch]

The main difference between the proof of the above result and the proof of the original formula is that, in principle, the initial state of an infinite number of vertices might influence the occurrence of~$A$. But we know from Lemma~\ref{l:decouple} that the probability of the state of a far away vertex influencing what happens inside~$K$ goes to zero faster than exponentially in the distance between said vertex and~$K$. This observation and a limiting argument are enough to conclude the result.

\end{proof}

\paragraph{An OSSS inequality for the model.} Given an event~$F$ depending on the states of the process at time~$\t$, we can write
\begin{equation}\label{eq:event}
  {\bf 1}_F = f \Big( \big( \sigma_{\rho, 0}(x) \big)_{x \in \Z^2}, (D_i)_{i \geq 0}, \mu' \Big),
\end{equation}
for some function $f$ with appropriate domain. As discussed in Section~\ref{s:proof_idea}, we want to define an algorithm~$\mathcal{A}$ that determines~$f$.

In order to use the powerful machinery related to the study of Boolean functions, we will let the algorithm know $\mu'$ from the beginning, having only to reveal the Bernoulli variables.
Given $\mu'$, the above function $f(\cdot, \cdot, \mu')$ can then be regarded as a Boolean function.

Define revealments~$\delta_\rho(\mathcal{A}, x, \mu')$ and $\delta_\rho(\mathcal{A}, i, \mu')$, in the same manner as in~\eqref{e:rev}: $\delta_\rho(\mathcal{A}, x, \mu')$ as the probability that the algorithm~$\mathcal{A}$ reveals the variable~$\sigma_{\rho, 0}(x)$ before stopping, and $\delta_\rho(\mathcal{A}, i, \mu')$ as the probability that the algorithm reveals~$D_i$ before stopping. We note that $\rho$ appears in the definition of the entries for $f$.

Considering~$\P_\rho, \E_\rho$ and $\Var_\rho$ as relating to the law of
\begin{equation}
  \Big( \big( \sigma_{\rho, 0}(x) \big)_{x \in \mathbb{Z}^2}, (D_i)_{i \geq 0}, \mu' \Big),
\end{equation}
we obtain the following result:
\begin{lemma}
  Let $F$ be an event depending on the states at time~$\t$ that is determined by~$(\sigma_{\rho, \t}(x))_{x \in K}$ for a finite set~$K \subset \Z^2$. Assume that $f$ is defined as~\ref{eq:event}. Then for any algorithm $\mathcal{A}$ that determines $f$, for $P$-almost every $\mu'$,
  \begin{equation}
    \label{e:OSSS_model}
    \Var_\rho( F | \mu') \leq 4 \bigg( \sum_{x \in \mathbb{Z}^2} \Inf^{\textnormal{init}}_\rho (x, F | \mu') \delta_\rho (\mathcal{A}, x, \mu')
    + \sum_{i \geq 0} \Inf^{\textnormal{thin}}_\rho (i, F | \mu') \delta_\rho (\mathcal{A}, i, \mu') \bigg).
  \end{equation}
\end{lemma}
The proof of this lemma follows immediately from Theorem~\ref{t:OSSS} once one realizes that, $\mu'$-a.s., there are only finitely many variables of the type~$\sigma_{\rho, 0}(x)$ and~$D_i$ that can affect the outcome of~$f$. The factor~$4$ in the RHS of the above equation comes from the fact that~$f$ takes values in~$\{0, 1\}$.

\begin{remark}
  Observe that the influences of both the initial configuration and the thinning variables appear in the upper bound of the variance.
  In Section~\ref{s:AG}, we will bound the influence of the thinning variables in terms of the initial condition in order to show sharp thresholds with respect to $\rho$.
\end{remark}

\paragraph{The algorithm.} We write~$\big[y \overset{\t,\, k}\rightarrow x\big]$ for the event where there exists in~$\mu'$ a sequence of time marks smaller than~$\t$ in increasing order associated to points in a path from $y$ to~$x$. In other words, the event is similar to~$\big[y \overset{\t}\rightarrow x\big]$, but defined for the point measure with thickened clock ticks.

For every $x \in \mathbb{Z}^2$ and fixed $\mu'$, we can define $\mathcal{B}(x)$, the random ``support region of $x$'', as
\begin{equation}
  \label{e:support}
  \mathcal{B}(x) := \big\{ y; y \overset{\t,\, k}\rightarrow x \big\}.
\end{equation}

Recall that our main goal is to prove a sharp threshold for the occurrence of the event~$\mathcal{C}(3n, n)$. With that in mind, we fix the function~$f$ to be determined as~${\bf 1}_{\mathcal{C}(3n, n)}$ and will now describe the algorithm. We first note that this~$f$ depends only on~ $(\sigma_{\rho, \t}(x))_{x \in [0, 3n] \times [0, n]}$. We then define the \emph{good event for~$x$}:
\begin{equation}
  \label{e:good_event}
  \mathcal{G}_x = \big[ \mu'; \mathcal{B}(x) \subseteq B(x, \log n) \big],
\end{equation}
and the \emph{good box event}
\begin{equation}
  \label{e:good_box_event}
  \mathcal{G}(3n, n) = \bigcap_{x \in [0, 3n] \times [0, n]} \mathcal{G}_x.
\end{equation}
We recall that our algorithm will start already knowing~$\mu'$. Noting that the good box event~$\mathcal{G}(3n, n)$ depends only on~$\mu'$, we start the algorithm \emph{already knowing if this event happens or not}. Our first step is asking if we are indeed in the good event.
\begin{itemize}
\item[(i)] If~$\mathcal{G}(3n, n)$ does not occur, we reveal the state of every relevant random variable: the initial states~$U^y$ and the thinning variables~$D_i$ associated to points in~$\cup_{x \in [0, 3n] \times [0, n]} \mathcal{B}(x)$.
\end{itemize}

We will need a sub-algorithm~$\mathcal{A}_x$ to query the value of $\sigma_{\rho, \t}(x)$. This algorithm, which will only be called if we already know that~$\mathcal{G}_x$ occurred, performs the actions:
\begin{itemize}
\item[$x$.i] Reveal every $U^y$ for $y \in \mathcal{B}(x)$;

\item[$x$.ii] Reveal every $D_i$ for which $X_i \in \mathcal{B}(x)$ and $T_i \in [0, \t]$.
\end{itemize}
Note that the above is enough to determine the value of $\sigma_{\rho, \t}(x)$.

We can now define the algorithm~$\mathcal{A}$ in a similar way to the algorithm defined in Section~\ref{s:proof_idea} to explore crossings in Bernoulli percolation:
\begin{itemize}
\item[(ii)] Choose an integer~$Z$ uniformly among~$\{0, \dots, 3n\}$;

\item[(iii)] Explore, in a previously defined order, the connected components of vertices with state~$+1$ in~$\sigma_{\rho, \t}$ intersecting the line~$\{Z\} \times [0, n]$. Whenever the state of a vertex~$x$ needs to be queried, we use~$\mathcal{A}_x$.
\end{itemize}
Since every crossing of $[0, 3n] \times [0, n]$ by pluses must cross the randomly selected line, the algorithm~$\mathcal{A}$ will determine the occurrence of~$\mathcal{C}(3n, n)$.

\paragraph{Estimating the revealment.} Consider the event where the state of a given variable associated to a site~$x$ is revealed by the algorithm. This implies that~$x$ belongs to the support of a point~$y$ (not necessarily different from~$x$) which is connected to the random line selected by the algorithm by a path of vertices with plus states. In particular, when~$x$ is far from this line, we get the occurrence of an~$\Arm$ event. With that in mind, we state the following lemma, which will help us estimate the revealment of the relevant variables.

\nck{k:one_arm}
\begin{lemma}
  \label{l:quenched_one_arm}
  There exists an exponent $\nu > 0$ such that, for all $\gamma > 0$, there exists $\uck{k:one_arm} = \uck{k:one_arm}(\gamma) \geq 2$ such that, for any $k \geq \uck{k:one_arm}$ and for any $\rho \leq \rho_+$ and every $n \geq n_0(k)$,
  \begin{equation}
    \label{e:quenched_one_arm}
    \P \Big[ \P_\rho \big[ \Arm (n^{1/4}, n^{1/2}) \big| \mu' \big] \geq n^{-\nu} \Big] \leq n^{-\gamma}.
  \end{equation}
\end{lemma}

\begin{remark}
  We note that the outer probability in the RHS of \eqref{e:quenched_one_arm} gives only~$\mu'$, and conditioned on~$\mu'$ the probability space is an almost surely finite hypercube. Intuitively speaking, this lemma says that the Poisson processes $\mu'$ for which the quenched arm event does not decay as a small power of $n$ are very rare.
\end{remark}

\begin{proof}[Proof idea]
  The proof of this result follows mutatus mutandi from the proof of Proposition $3.4$ of~\cite{alves2019sharp}: the Poisson point processes are slightly changed, but the argument is the same. The main idea is to decompose the annulus with inner radius~$n^{1/4}$ and outer radius~$n^{1/2}$ into approximately $\log n$ annuli with fixed aspect ratio. Crossing each of these annuli when~$\rho \leq \rho_{+}$ costs a fixed amount, so that multiplying all the costs yield a polynomial probability in~$n$. Dependence prevents one from straight up multiplying these probabilities, but taking these annuli sufficiently spaced and using a decoupling argument finishes the proof.
\end{proof}

We can now finally estimate the revealment of the relevant random variables.
\nck{k:revealment}
\begin{lemma}
  \label{l:revealment}
  There exists an exponent $\nu > 0$ and some $\uck{k:revealment} \geq 2$ such that for any $k \geq \uck{k:revealment}$ and for any $\rho \leq \rho_+$, we have
  \begin{equation}
    \label{e:revealment_x}
    \P \big[ \sup_{x \in \Z^2} \delta_\rho (\mathcal{A}, x, \mu') \geq n^{-\nu} \big] \leq n^{-50},
  \end{equation}
  \begin{equation}
    \label{e:revealment_i}
    \P \big[ \sup_{i \in \N_0} \delta_\rho (\mathcal{A}, i, \mu') \geq n^{-\nu} \big] \leq n^{-50},
  \end{equation}
  for every $n \geq n_0(k)$.
\end{lemma}
\begin{proof}
  Both thinning and initial variables are only queried if their associated point in~$\Z^2$ belongs to the support~$\mathcal{B}(x)$ of a point connected to the line selected by the algorithm. We will prove~\eqref{e:revealment_x}, the other bound follows analogously.

  We will consider two \emph{bad events}, the first of which is the complement of~$\mathcal{G}(3n, n)$, being the event where \emph{some} vertex in~$[0, 3n] \times [0, n]$ has large support.
  The second of the bad events happens when, conditioned on~$\mu'$, there exists a point of~$[0, 3n] \times [0, n]$ for which an associated arm event has large probability. Let~$\nu' > 0$ be such that Lemma~\ref{l:quenched_one_arm} is valid with~$\gamma = 100$, then define
  \begin{equation}
    \label{eq:revealment_large_arm}
    \mathrm{BadArm}_n
      := \Bigg\{
                \begin{array}{c}
                  \text{there exists } x \in [0, 3n] \times [0, n] \text{ such that}
                  \\
                  \P_\rho \big[ \Arm_x (n^{1/4}, n^{1/2}) \big| \mu' \big] \geq n^{-\nu'}
                \end{array}
              \Bigg\}.
  \end{equation}
  We assume~$k \geq k_0(100)$ given by Lemma~\ref{l:quenched_one_arm}, which together with the union bound implies
  \begin{equation}
    \label{eq:revealment_large_arm_estimate}
    \P \big[ \mathrm{BadArm}_n \big]
      \leq
        3 n^2 n^{-100}
      \leq
        3 n^{-98}.
  \end{equation}
  Lemma \ref{l:light} then implies, again together with an union bound,
  \begin{equation}
    \label{eq:revealment_large_support_estimate}
    \P \big[ \mathcal{G}(3n, n)^{c} \big]
      \leq
        \uc{c:light}(k \t) 3 n^2
        \exp \bigg\{  - \frac{1}{4} \log(n) \log (\log(n)) \bigg\}.
  \end{equation}

  We now bound the revealment on the event~$\mathcal{G}(3n, n) \cap \mathrm{BadArm}_n^c$. Either the vertex~$x$ is at a distance smaller than~$2 \sqrt{n}$ of the line randomly selected by the algorithm, or some vertex in the support of~$x$ connects to this randomly selected line via a $+$-path. Both are very unlikely on~$\mathcal{G}(3n, n) \cap \mathrm{BadArm}_n^c$, the first case is bounded from above by $2 \sqrt{n} / 3n$, the second by the probability of~$\mathrm{Arm}_x(n^{1/4}, n^{1/2})$. We have
  \begin{equation}
    \label{eq:revealment_small_1}
    \begin{split}
    \lefteqn{
            \delta_{\rho}(\mathcal{A}, x, \mu'){\bf 1}_{\mathcal{G}(3n, n) \cap \mathrm{BadArm}_n^c}
            }\quad
      \\ &\leq
        \Bigg(
            \sup_{x \in [0, 3n] \times n}
            \P_\rho \big[ \Arm_x (n^{1/4}, n^{1/2}) \big| \mu' \big]
            \Bigg)
            {\bf 1}_{\mathcal{G}(3n, n)\cap \mathrm{BadArm}_n^c}
        + \frac{1}{\sqrt{n}}
      \\ &\leq
        n^{-\nu'} + n^{-1/2}
      \\ &\leq
        n^{-\nu},
    \end{split}
  \end{equation}
  after choosing~$\nu$ sufficiently small. We have then proved that
  \begin{equation}
    \label{eq:revealment_small_2}
    \begin{split}
      \P \big[ \sup_{x \in \Z^2} \delta_\rho (\mathcal{A}, x, \mu') \geq n^{-\nu} \big]
      &\leq
        \P \big[ \mathcal{G}(3n, n)^{c} \cup \mathrm{BadArm}_n \big]
      \\ &\leq
        \uc{c:light}(k \t) 3 n^2
        \exp \bigg\{  - \frac{1}{4} \log(n) \log (\log(n)) \bigg\}
        +
        3 n^{-98}
      \\ &\leq
      n^{-50},
    \end{split}
  \end{equation}
  for sufficiently large~$n$, after choosing~$\nu$ appropriately. This finishes the proof of the result.
\end{proof}

\section{Comparison of pivotal probabilities}
\label{s:AG}
~
\par In this section we prove Lemma~\ref{l:AG}, which establishes a relation between the two concepts of influence introduced in Section~\ref{s:proof_ingredients}, namely, the influence of an initial opinion and the influence of a clock tick.
This will rely on an Aizenman-Grimmett type of argument, where the pivotality of a clock tick is transferred to pivotality of initial conditions through local changes in the quenched configuration.

Since each tick has a random region of influence, it is not necessarily true that we can relate its influence to the initial condition of its corresponding site. In fact, we prove that if there exists a pivotal clock tick, then it is possible to modify the configuration in a neighborhood of the tick and obtain a pivotal initial opinion on a nearby site.

Before stating our result, we need some additional notation. We define inductively $\log^{(1)}(j) = \log (j)$ and, for $k \in \N$, $\log^{(k+1)}(j) = \log \big(\log^{(k)}(j)\big)$. Given $j \geq 1$, let
\begin{equation}
  \label{e:sublog}
  h_{\alpha}(j) = \left\lceil \alpha \frac{\log(j)}{\log^{(2)} (j)} \right\rceil,
\end{equation}
where $\alpha \in \R_{+}$ is a value that will be chosen large enough afterwards.
Recall~\eqref{e:support}, the definition of the support $\mathcal{B}(x)$ of a vertex $x$, and consider the event
\begin{equation}\label{eq:omega_x}
\tilde{\mathcal{G}}_x^{k}(\alpha, \zeta) = \bigg\{ \mu'; \begin{array}{c}{ \mathcal{B}(x) \subset B(x, h_{\alpha}(n)) \text{ and } } \\ { \mu' \big(\partial B(x, h_{\alpha}(n)) \times [0,\t] \times [0,1] \big) \leq \zeta\log n} \end{array} \bigg\},
\end{equation}
as well as the good event
\begin{equation}
  \label{e:Omega_prime}
  \Omega(k, n) = \bigcap_{x \in [-n, 4n] \times [-n, 2n]} \tilde{\mathcal{G}}^{k}_x(\alpha,\zeta).
\end{equation}
Even though the event above depends on the values of $\alpha$ and $\zeta$, we suppress them in order to make notation cleaner.

The event $\Omega(k,n)$ is composed of realizations $\mu'$ such that there are no large supports on $[-n, 4n] \times [-n, 2n]$ and that the number of clock ticks before time $\t$ in the shells $\partial B(x, h_{\alpha}(n))$ is bounded. Furthermore, since $\mathcal{B}(x) \subset B(x, h_{\alpha}(n))$, for all $x \in [-n, 4n] \times [-n, 2n]$, we obtain that the opinion at time $\t$ of any $x \in [0,3n] \times [0,n]$ is determined by the initial opinions and clock selections inside $B(x, h_{\alpha}(n))$.

\nc{c:bound_omega}

Let us now bound the probability of ${\Omega(k, n)}^c$. We begin by bounding the probability that a given vertex has a large support. From Lemma~\ref{l:light}, we immediately get
\begin{equation}\label{eq:bound_omega_1}
\begin{split}
  \P\Big( \mathcal{B}(x) \cap B(x, h_{\alpha}(n))^{c} \neq \emptyset \Big) & \leq \sum_{y \notin B(x, h_{\alpha}(n))}\P\big( y \overset{\t, k} \to x \big)
  \leq \sum_{\ell \geq h_{\alpha}(n)}8\uc{c:light}(k \t) \ell e^{- \frac{1}{4}\ell \log \ell} \\
  & \leq e^{-c h_{\alpha}(n) \log h_{\alpha}(n)}
  \leq e^{-\uc{c:bound_omega} \alpha \log n},
\end{split}
\end{equation}
for some $\uc{c:bound_omega}=\uc{c:bound_omega}(k, \t)>0$.

We now proceed to bound the probability that a shell $\partial B(x, h_{\alpha}(n))$ has many clock ticks. Notice that the random variable $\mu' \big(\partial B(x, h_{\alpha}(n)) \times [0,\t] \times [0,1] \big)$ has distribution $\poisson\big( 8k \t h_{\alpha}(n) \big)$. From this, we can bound, for any $\lambda>0$,
\begin{equation}
\begin{split}
\P\Big( \mu' \big(\partial & B(x, h_{\alpha}(n)) \times [0,\t] \times[0,1] \big) > \zeta\log n \Big) \\
& \leq \exp\bigg\{ 8k \t h_{\alpha}(n) (e^{\lambda}-1) - \lambda \zeta \log n \bigg\}.
\end{split}
\end{equation}
We now take $\lambda = \log \Big( \frac{\zeta \log n}{8k \t h_{\alpha}(n)} \Big)>0$ to obtain
\begin{equation}\label{eq:bound_omega_2}
\begin{split}
  \P\Big( \mu' & \big(\partial B(x, h_{\alpha}(n)) \times [0,\t] \times [0,1] \big) > \zeta \log n \Big) \\
  & \leq \exp\bigg\{ \zeta \log n \Big( -\log \Big(\frac{\zeta \log n}{8k\t h_{\alpha}(n)} \Big)+1\Big) \bigg\} \\
  & \leq \exp\bigg\{\zeta\log n \Big( - \log^{(3)} n - \log \frac{\zeta}{16k \t \alpha} +1\Big) \bigg\}
  \leq e^{-\frac{\zeta}{2} \log n \log^{(3)}n},
\end{split}
\end{equation}
for all $n$ large enough, depending on $\zeta$, $k$, $\t$ and $\alpha$.

Combining Equations~\eqref{eq:bound_omega_1} and~\eqref{eq:bound_omega_2} with a union bound yields that, for any $k \geq 2$, there exists $\alpha$ large enough such that the following holds: for any $\zeta>0$, there exists $n_{0} \in \N$ such that
\begin{equation}\label{eq:bound_omega}
\P\big( \Omega(k, n)^{c} \big) \leq 30n^{2}\big( e^{-\uc{c:bound_omega} \alpha \log n } + e^{-\frac{\zeta}{2}\log n \log^{(3)} n} \big) \leq n^{-\frac{1}{2}},
\end{equation}
for all $n \geq n_{0}$.

The next lemma is the central piece that relates thinning influences to influences of initial opinions. In it, we prove that, provided we are in the good event $\Omega(k,n)$, the influence of a clock tick $i$ that happens in $X_{i} \in [-n, 4n] \times [-n, 2n]$ can be related to the influence of the initial opinion of vertices in $\partial B(X_{i}, h_{\alpha}(n))$ with a polynomially small correction.

\nc{c:AG}
\begin{lemma}
  \label{l:AG}
  Given $k \geq 2$, there exists a constant $\uc{c:AG} = \uc{c:AG}(k) > 0$ such that the following holds for any $\zeta>0$. Given $\mu' \in \Omega(k, n)$ and some $\rho \in (\rho_-, \rho_+)$, for every $i$ such that $T_i \leq \t$ and $X_i \in [-n, 4n] \times [-n, 2n] \cap \mathbb{Z}^2$, we have
  \begin{equation}
    \label{e:AG}
    \Inf^{\textnormal{thin}}_\rho \big( i, \mathcal{C}(3n, n) \big| \mu' \big) \leq \Bigg(\frac{\uc{c:AG}}{\rho(1 - \rho)}\Bigg)^{\zeta \log n} \sum_{x \in \partial B(X_i, h_{\alpha}(n))} \Inf^{\textnormal{init}}_\rho \big( x, \mathcal{C}(3n, n) \big| \mu' \big).
  \end{equation}
\end{lemma}

We first assume the statement of the lemma above and conclude the proof of sharpness. The idea is to prove that the derivative of $\P_{\rho} \big( \sigma_{\t} \in \mathcal{C}(3n, n)\big)$ has polynomial growth as a function of $n$ when restricted to $(\rho_{-}, \rho_{+})$. For this, we use the OSSS inequality together with the influence relation given by Lemma~\ref{l:AG} and Russo's formula to lower bound such derivative by a polynomial multiple of the variance. Proving that the variance is uniformly bounded from below will then conclude the proof. Theorem~\ref{t:phases} follows then from Lemma~\ref{l:fsc}.

\begin{proof}[Proof of Theorem~\ref{t:phases}]
  We start by proving that
  \begin{equation}
    \label{e:rho_equal}
    \rho_+ = \rho_-
  \end{equation}
  and for that we assume by contradiction that $\rho_- < \rho_+$.
  From this inequality, we can introduce an intermediary interval $\rho_- < \rho_-' < \rho_+' < \rho_+$.
  By the definition of $\rho_-$ and $\rho_+$, for all $\rho \in [\rho_-', \rho_+']$, we have
  \begin{equation}
    \P_\rho \big( \mathcal{C}(3n, n) \big) > \uc{c:rsw} \quad \text{and} \quad \P_\rho \big( \mathcal{C}^*(3n, n) \big) > \uc{c:rsw},
  \end{equation}
  for all $n \geq 1$.
  And moreover
  \begin{equation}
    \label{e:far_from_boundary}
    \rho ( 1 - \rho ) \geq \bar{c} > 0.
  \end{equation}

  Choose now $k$ large enough such that Lemma~\ref{l:revealment} applies and
  \begin{equation}\label{eq:choice_k}
    \frac{4}{k \uc{c:rsw}^{2}} \leq \frac{1}{4}.
  \end{equation}

  In particular, if we define
  \begin{equation}
    A = \bigg\{\P_{\rho}\big( \mathcal{C}(3n, n) \big| \mu' \big) \notin \Big(\frac{\uc{c:rsw}}{2}, 1-\frac{\uc{c:rsw}}{2}\Big) \bigg\},
  \end{equation}
  then~\eqref{e:bound_var_down} together with the choice of $k$ implies
  \begin{equation}\label{eq:bound_A}
    \P\big( A \big) \leq \frac{4}{k \uc{c:rsw}^{2}} \leq \frac{1}{4}.
  \end{equation}

  Second, let $\nu>0$ as in Lemma~\ref{l:revealment} and $\uc{c:AG}$ as in Lemma~\ref{l:AG}. Choose $\zeta>0$ such that
  \begin{equation}\label{eq:clock_ticks_constant_choice}
    \zeta \log (\uc{c:AG} / \bar{c}) \leq \frac{\nu}{2}.
  \end{equation}

  Define the event
  \begin{equation}
    B= \bigg\{\sup_{x}\big\{\delta_\rho (\mathcal{A}, x, \mu')\big\} \geq n^{-\nu} \text{ or } \sup_{i} \big\{ \delta_\rho (\mathcal{A}, i, \mu')\big\} \geq n^{-\nu} \bigg\},
  \end{equation}
  and notice that Lemma~\ref{l:revealment} implies
  \begin{equation}\label{eq:bound_B}
    \P\big( B \big) \leq 2n^{-50},
  \end{equation}
  assuming that $n$ is large enough.

  Assume now that we take $\mu' \in \big(A \cup B \big)^{c} \cap \Omega(k, n)$. From the fact that $\mu' \notin A$, we obtain the lower bound
  \begin{equation}\label{eq:sharpness_1}
    \Var_\rho( \mathcal{C}(3n, n) | \mu') = \P_{\rho}\big( \mathcal{C}(3n, n) \big| \mu' \big) \big[ 1- \P_{\rho}\big( \mathcal{C}(3n, n) \big| \mu' \big) \big] \geq \frac{\uc{c:rsw}^{2}}{4}.
  \end{equation}
  Now, since $\mu' \notin B$, we can apply the OSSS inequality to estimate
  \begin{equation}\label{eq:sharpness_2}
    \begin{split}
      \Var_\rho( & \mathcal{C}(3n, n) | \mu') \\
      & \leq \bigg( \sum_{x \in \mathbb{Z}^2} \Inf^{\textnormal{init}}_\rho (x, \mathcal{C}(3n, n)| \mu') \delta_\rho (\mathcal{A}, x, \mu')
      + \sum_{i \geq 0} \Inf^{\textnormal{thin}}_\rho (i, \mathcal{C}(3n, n)| \mu') \delta_\rho (\mathcal{A}, i, \mu') \bigg) \\
      & \leq n^{-\nu}\bigg(\sum_{x \in \mathbb{Z}^2} \Inf^{\textnormal{init}}_\rho (x, \mathcal{C}(3n, n)| \mu')+ \sum_{i \geq 0} \Inf^{\textnormal{thin}}_\rho (i, \mathcal{C}(3n, n)| \mu')\bigg).
    \end{split}
  \end{equation}
  We will now use Lemma~\ref{l:AG} together with the fact that $\mu' \in \Omega(k, n)$ to bound the last sum in the equation above. Notice first that $\Inf^{\textnormal{thin}}_\rho (i, \mathcal{C}(3n, n)| \mu')=0$, for all $i$ such that $X_{i} \notin [-n, 4n] \times [-n,2n]$ or $T_{i} \geq \t$. We abuse notation in the following, by omitting this restriction on the values of $i$. Lemma~\ref{l:AG} together with the choice of $\zeta$ and \eqref{e:far_from_boundary} readily imply
  \begin{equation}
    \begin{split}
      \sum_{i \geq 0} \Inf^{\textnormal{thin}}_\rho (i, \mathcal{C}(3n, n) | \mu') & \leq \Big( \frac{\uc{c:AG}}{\rho (1 - \rho)} \Big)^{\zeta \log n} \sum_{i} \sum_{x \in \partial B(X_i, h_{\alpha}(n))} \Inf^{\textnormal{init}}_\rho \big( x, \mathcal{C}(3n, n) \big| \mu' \big) \\
      & \leq n^{\frac{\nu}{2}} \sum_{i} \sum_{x \in \partial B(X_i, h_{\alpha}(n))} \Inf^{\textnormal{init}}_\rho \big( x, \mathcal{C}(3n, n) \big| \mu' \big).
    \end{split}
  \end{equation}
  Once again, since $\mu' \in \Omega(k, n)$, a counting argument implies that, for each $x \in [-n, 4n] \times [-n,2n]$, the number of times it appears in the last summation above is bounded by $8\zeta h_{\alpha}(n) \log n$ (notice that the term $8h_{\alpha}(n)$ comes from the choice of center for the balls where $x$ lies in the boundary), which gives the bound
  \begin{equation}\label{eq:sharpness_3}
    \sum_{i \geq 0} \Inf^{\textnormal{thin}}_\rho (i, \mathcal{C}(3n, n) | \mu') \leq 8\zeta h_{\alpha}(n) \big( \log n \big) n^{\frac{\nu}{2}}\sum_{x \in \mathbb{Z}^2} \Inf^{\textnormal{init}}_\rho (x, \mathcal{C}(3n, n) | \mu').
  \end{equation}

  Combining Equations~\eqref{eq:sharpness_1},~\eqref{eq:sharpness_2}, and~\eqref{eq:sharpness_3} with Russo's formula (Lemma~\ref{l:russo}) implies that, for each $\mu' \in \big(A \cup B \big)^{c} \cap \Omega(k, n)$,
  \begin{equation}
    \begin{split}
      \frac{\uc{c:rsw}^{2}}{4} & \leq n^{-\nu}\big(8\zeta h_{\alpha}(n) \big( \log n \big) n^{\frac{\nu}{2}}+1 \big)\sum_{x \in \mathbb{Z}^2} \Inf^{\textnormal{init}}_\rho \big( x, \mathcal{C}(3n, n) \big| \mu' \big) \\
      & = n^{-\nu}\big(8\zeta h_{\alpha}(n) \big( \log n \big) n^{\frac{\nu}{2}}+1 \big)\frac{\d}{\d \rho}\P_{\rho} \big( \sigma_{\t} \in \mathcal{C}(3n, n) \big| \mu' \big) \\
      & \leq n^{-\frac{\nu}{4}}\frac{\d}{\d \rho}\P_{\rho} \big( \sigma_{\t} \in \mathcal{C}(3n, n) \big| \mu' \big),
    \end{split}
  \end{equation}
  if $n$ is chosen large enough, since $h_{\alpha}$ has sub-logarithmic growth (see Equation~\eqref{e:sublog}). From this,
  \begin{equation}
    \frac{\d}{\d \rho}\P_{\rho} \big( \sigma_{\t} \in \mathcal{C}(3n, n) \big| \mu' \big) \geq \frac{\uc{c:rsw}^{2}}{4}n^{\frac{\nu}{4}},
  \end{equation}
  for all $\mu' \in \big(A \cup B \big)^{c} \cap \Omega(k, n)$.

  We can now estimate, for all $\rho \in [\rho_-', \rho_+']$,
  \begin{equation}
    \begin{split}
      \frac{\d}{\d \rho}\P_{\rho} \big( \sigma_{\t} \in \mathcal{C}(3n, n) \big) & = \frac{\d}{\d \rho}\E\Big[ \P_{\rho} \big( \sigma_{\t} \in \mathcal{C}(3n, n) \big| \mu' \big) \Big] \\
      & = \E\bigg[ \frac{\d}{\d \rho} \P_{\rho} \big( \sigma_{\t} \in \mathcal{C}(3n, n) \big| \mu' \big) \bigg] \\
      & \geq \E\bigg[ \frac{\d}{\d \rho} \P_{\rho} \big( \sigma_{\t} \in \mathcal{C}(3n, n) \big| \mu' \big) {\bf 1}_{(A \cup B)^{c} \cap \Omega(k, n)}(\mu')\bigg] \\
      & \geq \frac{\uc{c:rsw}^{2}}{4}n^{\frac{\nu}{4}}\P\Big(\big(A \cup B \big)^{c} \cap \Omega(k, n)\Big).
    \end{split}
  \end{equation}
  Notice that, in the second line above, the differentiation and expectation signs can be exchanged by a simple application of the bounded convergence theorem. Indeed, the quotients in the definition of the derivative can be bounded by the maximum number of pivotal vertices, which in turn is bounded by the number of time marks of $\mu'$ that can influence the opinions inside the box $\mathcal{C}(3n, n)$. This last quantity is easily seen to have finite expectation.

  We now combine~\eqref{eq:bound_omega},~\eqref{eq:bound_A}, and~\eqref{eq:bound_B} to conclude that
  \begin{equation}
    \begin{split}
      \P\Big(A \cup B \cup \Omega(k, n)^{c}\Big) & \leq \P\big( A \big) + \P\big( B \big) + \P\big( \Omega(k, n)^{c} \big) \\
      & \leq \frac{1}{4}+2n^{-50}+n^{-\frac{1}{2}} \leq \frac{1}{2}.
    \end{split}
  \end{equation}

  This implies that
  \begin{equation}
    \begin{split}
      \frac{\uc{c:rsw}^{2}}{8}n^{\frac{\nu}{4}}|\rho_+' - \rho_-'| & \leq \int_{\rho_-'}^{\rho_+'} \frac{\d}{\d \rho}\P _{\rho}\big( \sigma_{\t} \in \mathcal{C}(3n, n) \big) \d \rho \\
      & = \P_{\rho_+'} \big( \sigma_{\t} \in \mathcal{C}(3n, n)\big) - \P_{\rho_-'} \big( \sigma_{\t} \in \mathcal{C}(3n, n)\big)
      \leq 1
    \end{split}
  \end{equation}
  which yields $|\rho_+' - \rho_-'| \leq cn^{-\frac{\nu}{4}}$, for all $n$ large enough.
  This implies $\rho_-' = \rho_+'$ which is a contradiction to our choices of these constants.
  Consequently, we have proved $\rho_- = \rho_+$ as claimed in \eqref{e:rho_equal}.

  We now claim that $\rho_- = \rho_+ = \rho_c(\tau)$.
  Indeed, for any $\rho < \rho_-$ we use the definition \eqref{eq:rho_-} to conclude that for some $n \geq 1$, $\mathbb{P}_\rho (\mathcal{C}^*(3n, n)) > 1 - \uc{c:fsc}$.
  Thus, by Lemma~\ref{l:fsc} $\mathbb{P}(\mathcal{C}^*(3n, n)) \to 1$, implying that $\mathbb{P}_\rho (0 \leftrightarrow \infty) = 0$, therefore $\rho < \rho_c(\tau)$.

  Analogously, if $\rho > \rho_+$, then \eqref{eq:rho_+} gives that for some $n \geq 1$, $\mathbb{P}_\rho (\mathcal{C}(3n, n)) > 1 - \uc{c:fsc}$ and Lemma~\ref{l:fsc} implies that $\mathbb{P}_\rho (\mathcal{C}(3n, n)^c)$ is summable over $n$ of the form $3^k$.
  Therefore if we denote by $R_k = [0, 3^{k + 1}] \times [0, 3^k]$ and $R_k' = [0, 3^k] \times [0, 3^{k + 1}]$, we know by Borel-Cantelli's Lemma that all but finitely many of these rectangles are crossed in the hard direction.
  On this event, there exists a point which is connected to infinity, so that $\mathbb{P}_\rho (0 \leftrightarrow \infty) > 0$ and $\rho > \rho_c(\tau)$.
  This concludes the proof that $\rho_- = \rho_+ = \rho_c(\tau)$.

  We now finish the proof of Theorem~\ref{t:phases} and for this we start by observing that \eqref{e:critical} follows from \eqref{eq:rho_-}, \eqref{eq:rho_+} and FKG, while \eqref{e:fast_to_zero} and \eqref{e:fast_to_one} follow from Lemma~\ref{l:fsc} by appropriately choosing the constants
$\uc{c:sub_critical}$ and $\uc{c:super_critical}$.

  We finally prove \eqref{e:continuous}.
  Observe first that $\P_\rho [ 0 \overset{+}\leftrightarrow \partial B(0, n) ]$ is continuous and increasing for any $n \geq 1$.
  Therefore its decreasing limit $\P_\rho [ 0 \overset{+}\leftrightarrow \infty ]$ is right continuous for $\rho \in [0, 1]$.
  To prove that $\P_\rho [ 0 \overset{+}\leftrightarrow \infty ]$ is left continuous for $\rho \in (\rho_c, 1]$ we follow a classical argument from van den Berg and Keane for Bernoulli percolation, see also Section~8.3 of \cite{Gri99}.

  More precisely, it follows from \cite{gandolfi1988uniqueness} that the infinite cluster of $+$'s is unique for every $\rho$.
  For $\rho > \rho_c$, we denote by $\mathcal{C}^\rho_0$ the cluster of the origin at $\rho$ and by $\mathcal{C}^\rho$ the unique infinite cluster.
  Uniqueness implies that, given $\rho_c < \rho_1 < \rho_2$, $\mathcal{C}^{\rho_1} \subseteq \mathcal{C}^{\rho_2}$.

  On the event that $|\mathcal{C}^{\rho_2}_{0}| = \infty$, there exists a finite path from $0$ to $\mathcal{C}^{\rho_1}$ that is open at $\rho_2$.
  Since this path is finite, we know that it must be open for some value $\rho \in (\rho_1, \rho_2)$ and thus
  \begin{equation}
    \P \big[ |\mathcal{C}^{\rho_2}_0| = \infty, |\mathcal{C}^\rho| < \infty \text{ for all $\rho < \rho_2$} \big] = 0.
  \end{equation}
  This concludes that $\P_\rho [ 0 \overset{+}\leftrightarrow \infty ]$ is continuous over $(\rho_c, 1]$.

  All that is left to prove is that $\P_\rho [ 0 \overset{+}\leftrightarrow \infty ]$ is left-continuous at $\rho_c$.
  For this we use \eqref{e:critical} to note that $\P_{\rho_c} [ 0 \overset{+}\leftrightarrow \infty ] = 0$, finishing the proof of \eqref{e:continuous} and of Theorem~\ref{t:phases}.
\end{proof}

We now proceed to the proof of Lemma~\ref{l:AG}. Here, assuming that a clock tick $i$ is pivotal for a configuration $\mu' \in \Omega(k,n)$, we change the initial configuration and clock-tick selections on the shell $\partial B(X_{i}, h_{\alpha}(n))$ in order to obtain a vertex with pivotal initial condition. Due to the fact that $\mu' \in \Omega(k,n)$, we can bound the Radon-Nikodym derivative of this change of configurations and obtain estimates that lead to~\eqref{e:AG}.

\begin{proof}[Proof of Lemma~\ref{l:AG}]
Fix $\mu' \in \Omega(k, n)$. Recall that we write $\mu' = \sum_{j \geq 0} \delta_{(X_{j}, T_{j}, R_{j})}$, and denote by
\begin{equation}
\mathcal{S}= [-n, 4n] \times [-n, 2n]  \, \text{ and } \, \mathcal{T} = \Big\{j: (X_{j}, T_{j}) \in \mathcal{S} \times [0,\t] \Big\}
\end{equation}
the collections of sites and clock ticks whose selection can influence the configuration $\sigma_{\t}$ restricted to the rectangle $[0,3n]\times [0,n]$.

Assume that the clock tick $i$ is pivotal, meaning that its selection determines the existence of crossings at time $\t$. Without loss of generality, we will consider only the case when the acceptance of this clock tick yields a crossing of the rectangle $[0,3n]\times [0,n]$. Since $\mu' \in \Omega(k,n)$, we have $B(X_{i}, h_{\alpha}(n)) \in [-n,4n] \times [-n, 2n]$. Furthermore, the only vertices whose opinions can be changed through the selection of the tick $i$ are the ones in $B(X_{i}, h_{\alpha}(n))$, since the support of any $y \notin B(X_{i}, h_{\alpha}(n))$ cannot contain $X_{i}$.

We will now define maps $\varphi_{-}^{i}, \varphi_{+}^{i}: \{-1,1\}^{\mathcal{S} \cup \mathcal{T}} \to \{-1,1\}^{\mathcal{S} \cup \mathcal{T}}$ that locally modify the initial configuration and clock selections on the boundary $\partial B (X_{i}, h_{\alpha}(n))$ (we assume here, in order to ease the notation, that -1 in a clock tick means that it is not accepted, while a +1 signals its acceptance). The map $\varphi_{-}^{i}$ changes the initial opinion $\sigma_{0}|_{\partial B (X_{i}, h_{\alpha}(n))}$ to the constant $-1$ and forbids all clock-tick selections that happen before time $\t$ in any vertex of $\partial B (X_{i}, h_{\alpha}(n))$. Analogously, $\varphi_{+}^{i}$ changes the initial configuration on $\partial B (X_{i}, h_{\alpha}(n))$ to the all one and disregards the same clock ticks as in $\varphi_{-}^{i}$.

Given a selection configuration $\omega \in \{-1,1\}^{\mathcal{S} \cup \mathcal{T}}$ for which $i$ is pivotal, we now argue that $\sigma_{\t}(\varphi_{-}^{i}(\omega))$ does not contain a horizontal crossing of the rectangle $[0,3n] \times [0,n]$. In order to do so, it suffices to verify that $\sigma_{\t}(\varphi_{-}^{i}(\omega))(y) \leq \sigma_{\t}(\omega^{i,-})(y)$, for all $y \notin B(X_{i}, h_{\alpha}(n))$, where $\omega^{i,-}$ is obtained from $\omega$ by changing the selection of the clock tick $i$ to -1. This will readily imply the absence of crossings, since $\sigma_{\t}(\varphi_{-}^{i}(\omega))(y)=-1$, for all $y \in \partial B(X_{i}, h_{\alpha}(n))$, and $i$ is pivotal for $\omega$.

We proceed in two steps to go from $\omega^{i,-}$ to $\varphi_{-}^{i}(\omega)$. First, change the initial configuration for $\omega^{i,-}$ in $\partial B (X_{i}, h_{\alpha}(n))$ to the constant equal to -1 and obtain a intermediate configuration $\omega^{-}$ which satisfies $\sigma_{\t}(\omega^{-}) \prec \sigma_{\t}(\omega^{i,-})$ due to monotonicity with respect to the initial condition. Second, from $\omega^{-}$ we cancel all the clock ticks that happen before time $\t$ in $\partial B (X_{i}, h_{\alpha}(n))$. This can only decrease the configuration at time $\t$, since we remove the chance that an opinion on $\partial B (X_{i}, h_{\alpha}(n))$ would eventually turn to 1. We thus obtain the domination $\sigma_{\t}(\varphi_{-}^{i}(\omega^{i,-})) \prec \sigma_{\t}(\omega^{i,-})$. The proof of the claim is then completed by noticing that $\sigma_{\t}(\varphi_{-}^{i}(\omega^{i,-}))$ and $\sigma_{\t}(\varphi_{-}^{i}(\omega))$ coincide outside the ball $B(X_{i}, h_{\alpha}(n))$. A similar argument holds for $\sigma_{\t}(\omega^{i,+}) \prec \sigma_{\t}(\varphi_{+}^{i}(\omega))$, where $\omega^{i,+}$ is obtained from $\omega$ by changing the selection of the clock tick $i$ to 1.

From the discussion above, we obtain that, if the clock tick $i$ is pivotal for a configuration $\omega$, then there are two configurations that differ only in the initial condition, namely, $\varphi_{-}^{i}(\omega)$ and $\varphi_{+}^{i}(\omega)$, such that the initial opinions on the set $\partial B (X_{i}, h_{\alpha}(n))$ are jointly pivotal. This means that

\begin{display}
if $\sigma_{0}|_{\partial B (X_{i}, h_{\alpha}(n))} \equiv -1$, then there is no crossing, while a crossing is present at time $\t$ if $\sigma_{0}|_{\partial B (X_{i}, h_{\alpha}(n))} \equiv 1$.
\end{display}

Let us now bound the Radon-Nikodym derivative of this change of measure. Observe first that, due to the fact that $\mu' \in \Omega(k, n)$, in order to go from a configuration $\omega \in \{-1,1\}^{\mathcal{S} \cup \mathcal{T}}$ to $\varphi_{-}^{i}(\omega)$, one needs to change at most $8h_{\alpha}(n)$ initial positions and at most $\zeta \log n$ clock-tick selections. In particular, this implies that, if $\varphi_{-}^{i}(\omega) = \tilde{\omega}$,
\begin{equation}
\P_{\rho}\big(\omega|\mu'\big) \leq k^{\zeta \log n}\left(\frac{1}{1-\rho}\right)^{8h_{\alpha}(n)}\P_{\rho}\big(\tilde{\omega}\big|\mu'\big) \leq \left(\frac{k}{1-\rho}\right)^{\zeta \log n}\P_{\rho}\big(\tilde{\omega}\big|\mu'\big).
\end{equation}
Furthermore, for each $\tilde{\omega} \in \{-1,1\}^{\mathcal{S} \cup \mathcal{T}}$, there are at most $2^{8h_{\alpha}(n)}2^{\zeta \log n} \leq 4^{\zeta \log n}$ configurations $\omega \in \{-1,1\}^{\mathcal{S} \cup \mathcal{T}}$ such that $\varphi_{-}^{i}(\omega)=\tilde{\omega}$. From this, we can estimate
\begin{equation}\label{eq:bound_entropy}
\begin{split}
\Inf^{\textnormal{thin}}_\rho \big( & i, \mathcal{C}(3n, n) \big| \mu' \big) = \sum_{\omega \in \{-1,1\}^{\mathcal{S}\cup \mathcal{T}}}\P\big(\omega \big| \mu' \big){\bf 1}_{i \text{ is pivotal}}(\omega) \\
& \leq  \sum_{ \tilde{\omega} }\sum_{\omega \in (\varphi_{-}^{i})^{-1}(\tilde{\omega}) 	} \P\big(\omega \big| \mu' \big){\bf 1}_{\partial B (X_{i}, h_{\alpha}(n)) \text{ is jointly pivotal}}(\tilde{\omega}) \\
& \leq 4^{\zeta \log n} \left(\frac{k}{1-\rho}\right)^{\zeta \log n}\sum_{\tilde{\omega} \in \{-1,1\}^{\mathcal{S} \cup \mathcal{T}}}\P_{\rho}\big(\tilde{\omega}\big|\mu'\big){\bf 1}_{\partial B (X_{i}, h_{\alpha}(n)) \text{ is jointly pivotal}}(\tilde{\omega}) \\
& \leq \left(\frac{4k}{1-\rho}\right)^{\zeta \log n}\P_{\rho}\big( \partial B (X_{i}, h_{\alpha}(n)) \text{ is jointly pivotal} \big| \mu' \big).
\end{split}
\end{equation}

Finally, let us consider the event on the RHS of the equation above. Assume that $\partial B (X_{i}, h_{\alpha}(n))$ is pivotal for a configuration $\omega \in \{-1,1\}^{\mathcal{S} \cup \mathcal{T}}$ and that $\sigma_{\t}(\omega)$ does not have a crossing. If this happens, we can progressively change the initial opinions of $\omega$ in $\partial B (X_{i}, h_{\alpha}(n))$ from -1 to 1 and eventually arrive at a configuration with one such vertex being pivotal. This implies
\begin{equation}
\begin{split}
\P_{\rho}\big( \partial B (X_{i}, h_{\alpha}(n)) & \text{ is closed pivotal} \big| \mu' \big) \\
& \leq \left(\frac{2}{\rho}\right)^{8h_{\alpha}(n)}\sum_{x \in \partial B (X_{i}, h_{\alpha}(n))} \Inf^{\textnormal{init}}_\rho \big( x, \mathcal{C}(3n, n) \big| \mu' \big).
\end{split}
\end{equation}
This finishes the proof of the desired inequality. An analogous bound holds for the case when $\sigma_{\t}(\omega)$ has a crossing, but we change the vertices whose opinions are 1.
\end{proof}

\begin{remark}
  \label{r:higher_dimensions}
  We briefly discuss the reasons why our proof for sharp thresholds cannot be extended to higher dimensions. In order to obtain a suitable differential inequality via the OSSS inequality and Russo's formula, we considered the graphical representation conditioned on the Poissonian clock ticks in Section \ref{s:model_and_notations}. This made us lose crucial symmetries of the process, which we recovered using a \emph{homogeneization by thickening} argument (Lemma \ref{l:thinning}). We still needed to get rid of the terms involving the influence of clock ticks appearing in the OSSS inequality. This required a surgery which had exponential price on the size of a boundary of a box with radius $\sim \varepsilon \log n$ (Lemma \ref{l:AG}). Using RSW theory, a planar-specific argument, we then managed to bound the revealment of a site using a polynomial bound for the one-arm event (Lemma \ref{l:revealment}). It is feasible to imagine that in higher dimensions one could use a different differential inequaility, such as $(13)$ of \cite{D-CRT19}, and forgo planarity arguments. Crucially for our case, however, the above surgery price is sub-polynomial in dimension two, and super-polynomial in higher dimensions (see Equation \eqref{eq:bound_entropy}), which makes such an argument impossible.
\end{remark}

\section{Elliptic bootstrap percolation}\label{s:elliptic_bootstrap}
~
\par In this section, we consider elliptic bootstrap percolation and point out how our proof can be adapted for this case. Let us first start by precisely defining the model we consider.

\bigskip

\noindent{\bf Definition of the model.} We will denote the configuration at time $t \geq 0$ by $\sigma_{t} \in \{-1,1\}^{\Z^{2}}$. In order to match with the rest of the paper, we set the unusual choice of opinions to be -1 and 1.

Consider the rate function
\begin{equation}
\lambda:\{-1,1\}^{B(0,1)} \to \R_{+},
\end{equation}
where $B(0,1)=\{x \in \Z^{2}: ||x||_{1} \leq 1\}$.
This function will control the update rate of any site in the process. We assume that $\lambda$ is a rotationally-invariant monotone non-decreasing function and that $\lambda(-\bar{1}) = \chi > 0$, where $-\bar{1}$ denotes the constant configuration equal to $-1$.

We set the evolution of the process in the following way. At time zero, define $\sigma_{0}(x) = -1$, for all $x \in \Z^{2}$. A site $x \in \Z^{2}$ turns from -1 to 1 with rate $\lambda(\t_{x}(\sigma_{t}))$, where $\t_{x}: \Z^{2} \to \Z^{2}$ is the translation $\t_{x}(y) = y-x$. A site with opinion 1 never changes back to -1.

The evolution can be described via its infinitesimal generator: for any local function $f: \{-1,1\}^{\Z^{2}} \to \R$, we define
\begin{equation}\label{eq:generator}
\mathcal{L}f(\sigma) = \sum_{x \in \Z^{2}}\lambda(\tau_{x}(\sigma))\big(f(\sigma^{x})-f(\sigma) \big),
\end{equation}
where
\begin{equation}
\sigma^{x}(y) = \begin{cases}
1, & \text{if } y=x, \\
\sigma(y), & \text{otherwise.}
\end{cases}
\end{equation}

By performing a time change, we can assume without loss of generality that $\lambda(\bar{1}) = 1$. We will denote by $\P_{t}$ the distribution of $\sigma_{t}$ when $\sigma_{0}=-\bar{1}$.

\bigskip

\noindent{\bf Percolation.} Due to the fact that $\lambda(\sigma) \in [\chi, 1]$, for all $\sigma \in \{-1,1\}^{B(0,1)}$,the configuration $\sigma_{t}$ dominates and is dominated by independent Bernoulli percolation with parameters $p_{t}=1-e^{\chi t}$ and $q_{t}=1-e^{-t}$, respectively. In particular, this implies that the critical percolation time
\begin{equation}
  \label{e:t_c}
t_{c} = \inf\big\{t \geq 0 : \P\big( \sigma_{t} \text{ percolates} \big)>0 \big\}
\end{equation}
is well defined, finite and positive.
For each $t \geq 0$ fixed, the distribution of $\sigma_{t}$ is translation and rotation invariant. Positive assossiation once again follows Corollary~1.2 from~\cite{harris1977} and Proposition~9.3 from \cite{sullivan75}.
Moreover, Lemmas~\ref{l:decouple} and~\ref{l:fsc} remain valid.

Our goal here is to prove that the process undergoes a sharp phase transition as time increases. Analogously to~\eqref{eq:rho_+} and~\eqref{eq:rho_-}, define the parameters
\begin{equation}\label{eq:t_+}
    \begin{split}
    t_{+} & = \inf \Big\{ t \in \R_{+}; \P_{t} (\mathcal{C}(3n, n)) > 1 - \uc{c:fsc} \text{ for some $n \geq 1$} \Big\}\\
    & = \sup \Big\{ t \in \R_{+}; \P_{t} (\mathcal{C}^*(n, 3n)) \geq \uc{c:fsc} \text{ for all $n \geq 1$} \Big\},
  \end{split}
  \end{equation}
  \begin{equation}\label{eq:t_-}
    \begin{split}
    t_{-} & = \sup \Big\{ t \in \R_{+} ; \P_{t} (\mathcal{C}^*(3n, n)) > 1 - \uc{c:fsc} \text{ for some $n \geq 1$} \Big\}\\
    & = \inf \Big\{ t \in \R_{+} ; \P_{t} (\mathcal{C}(n, 3n)) \geq \uc{c:fsc} \text{ for all $n \geq 1$} \Big\}.
  \end{split}
  \end{equation}

\nc{c:rsw_bootstrap}

Since our process dominates independent Bernoulli percolation with parameter $1-e^{-\chi t}$, we have $0 < t_{-} \leq t_{+} < \infty$ and $t_{c} \in [t_{-}, t_{+}]$. Furthermore, since $t_{+}$ is finite, we fix from now on a finite time $\t > t_{+}$ and restrict our analysis to times $t \in [0,\t]$. Similarly as in the previous case, there exists $\uc{c:rsw_bootstrap}>0$ such that, for all $t \in (t_{-}, t_{+})$
  \begin{display}
    $\P_t \big( \mathcal{C}(3n, n) \big) > \uc{c:rsw_bootstrap}$ and
    $\P_t \big( \mathcal{C}^*(3n, n) \big) > \uc{c:rsw_bootstrap}$, for all $n \geq 1$.
  \end{display}

The strategy for proving the Theorem~\ref{t:phases_bootstrap} above is the same as the one presented for Theorem~\ref{t:phases}. However, instead of carrying out the proof completely, we will only state the modifications on the main tools and focus more heavily on the Aizenman-Grimmett argument that relates the two different notions of pivolatily we have here.

\bigskip

\noindent{\bf Graphical construction.} Let us begin by defining the graphical construction that is used for this process. We will skip directly to the construction with the denser collection of clock rings.

As in Section~\ref{s:proof_ingredients}, we consider a Poisson point process on $\mathbb{Z}^2 \times \mathbb{R}_+ \times [0, 1] \times \{0, 1\}$ with intensity $\Count_{\mathbb{Z}^2} \otimes k \cdot \d x \otimes U[0,1] \otimes \Ber(1/k)$. We will continue to write  $\mu = \sum_{i \geq 0} \delta_{(X_i, T_i, R_i, D_i)}$, but we change the notation $\mu' = \sum_{i \geq 0} \delta_{(X_i, T_i)}$.

Each point $(X_{i}, T_{i}, R_{i}, D_{i})$ encodes the same as in the previous case:
  \begin{itemize}
  \item $X_{i} \in \mathbb{Z}^2$ will represent the site where the time tick occurs;
  \item $T_{i} \in \mathbb{R}_+$ will encode the time when each mark arrives;
  \item $R_{i} \in [0, 1]$ will hold the uniform random variable that is used in order to apply the rules of the local updates;
  \item $D_{i} \in \{0, 1\}$ is a variable used to thin the point process.
  \end{itemize}
Suppose a clock tick $(X_{i},T_{i},R_{i},D_{i})$ happens. Then the opinion in the site $X_{i}$ changes to 1 at time $T_{i}$ if $R_{i} \leq \lambda \big(\tau_{X_{i}} (\sigma_{T_{i}-}) \big)$ and $D_{i}=1$.

Notice that Lemma~\ref{l:light} still holds with this construction, since the only change in the construction of the dynamics for this case is in the updating rule, while the Poisson point process considered still remains the same.

\bigskip

\noindent{\bf Pivotality and Russo's formula.} We cannot obtain an explicit formula for the derivative of the crossing probabilities with respect to time. However, combining the expression of the generator with the fact that the rate function $\lambda$ satisfies $\lambda \geq \chi$, we can lower bound such derivative. For a monotone non-decreasing event $A$ depending on the states of vertices in a finite subset $K \in \Z^{2}$, we have
\begin{equation}
\begin{split}
  \frac{\d}{\d t} \P \big( \sigma_{t} \in A \big) & = \sum_{x \in K} \E\Big(\lambda\big(\tau_{x}(\sigma_{t})\big) \big( {\bf 1}_{A}(\sigma_{t}^{x}) - {\bf 1}_{A}(\sigma_{t})\big)\Big) \\
  & = \sum_{x \in K} \E\Big(\lambda\big(\tau_{x}(\sigma_{t})\big) {\bf 1}_{A}(\sigma_{t}^{x}){\bf 1}_{A^{c}}(\sigma_{t})\Big) \\
  & \geq \chi \sum_{x \in K} \P\big( \sigma_{t} \notin A, \sigma_{t}^{x} \in A \big).
\end{split}
\end{equation}

In the bound above, the notion of pivotality that appears is different from pivotality for the initial condition introduced in Section~\ref{s:proof_ingredients}. We here have closed pivotailty with respect to the configuration at time $t$.

\begin{definition}
We say that $x \in \Z^{2}$ is closed time-$t$ pivotal for the monotone non-decreasing event $A$ if $\sigma_{t} \notin A$ and $\sigma_{t}^{x} \in A$.
\end{definition}

In particular, closed time-$t$ pivotality can only happen if $\sigma_{t}(x)=-1$. We define the conditional time-$t$ influence of a vertex $x \in \Z^{2}$ as
\begin{equation}
  \Inf^{t} (x, A | \mu') = \P \big(x \text{ is closed time-$t$ pivotal for $A$} \big| \mu'\big).
\end{equation}

We will also consider the conditional thinning influence, but we need to slightly change the definition. Fix a realization $\mu'$ and $i \geq 0$. Denote by $\sigma_{t}^{i,0}$ and $\sigma_{t}^{i,1}$ the configurations obtained by setting $D_{i}$ to 0 or 1, respectively.
\begin{definition}
We say that the clock tick $i$ is pivotal for the non-decreasing event $A$ if $\sigma_{t}^{i,1} \in A$ and $\sigma_{t}^{i,0} \notin A$, for some choice of $R_{i}$.
\end{definition}
In the above, since we are dealing with non-decreasing events $A$, it suffices to choose $R_{i}=0$. Furthermore, we define the influence of a clock tick $i$ as
\begin{equation}
  \Inf^{\textnormal{thin}} (i, A | \mu') = \P [i \text{ is pivotal for $A$} | \mu'].
\end{equation}

Once again, we can consider randomized algorithms $\mathcal{A}$ that determine a function $f={\bf 1}_{A}$. In our new setting, the algorithm has access to $\mu'$ and has to examine the variables $(R_{i}, D_{i})_{i \geq 0}$ in order to determine the occurrence of the event $A$. In this setting, we let $\mu'$ forget the values of $R_{i}$, as they will be used on our enhancement argument. Of course, the influence of a random variable\footnote{Here, influence is defined as the probability that resampling only this random variable changes the outcome of the crossing function we are considering} $R_{i}$ will be dominated by the influence of the corresponding $D_{i}$, since $R_{i}$ can only be pivotal if $D_{i}$ is. From this, we obtain the following version of the OSSS inequality
\begin{lemma}
  For any algorithm $\mathcal{A}$ that determines $f={\bf 1}_{A}$,
  \begin{equation}
    \label{e:OSSS_bootstrap}
    \Var( f | \mu') \leq 2 \sum_{i \geq 0} \delta (\mathcal{A}, i) \Inf^{\textnormal{thin}} (i, A | \mu').
  \end{equation}
\end{lemma}

Notice that, contrary to~\eqref{e:OSSS_model}, here we do not have an additional factor of 4 in the lemma above since our function takes values in $\{0,1\}$ instead of $\{-1, 1\}$.

\bigskip

\noindent{\bf Algorithm and revealment bounds.} We now proceed to introduce the algorithm and bound its revealment. This will follow the same lines as the previous case, and we just list the properties pointing out where modifications are necessary.

We start by recalling the definition of support of a site in~\eqref{e:support} and the good event~\eqref{e:good_event}. Once again, the support of a given vertex $x$ depends only on the realization of the Poisson point process $\mu'$ and the value of $\sigma_{t}(x)$ is determined if all clock-tick selections in the support of $x$ are revealed using the suitable modification of the algorithm $\mathcal{A}_{x}$. We can then apply the same algorithm in order to determine the existence of crossings of $\mathcal{C}(3n,n)$ at time $t<\t$.

Here once again, Lemma~\ref{l:revealment} can be applied to estimate the revealment and we obtain an analogous result.

\bigskip

\noindent{\bf Aizenman-Grimmett.} The last ingredient we need is one analogous to Lemma~\ref{l:AG}, which we state and prove here.

For $x \in \Z^{2}$ and $\ell \in \N$, denote by $J_{x}(\ell)$ the set
\begin{equation}
J_{x}(\ell) = \big\{j \in \N: (X_{j}, T_{j}) \in \partial B(x, \ell) \times [0,\t] \big\}.
\end{equation}

\nc{c:AG_bootstrap}

Recall the definition of $h_{\alpha}$ in~\eqref{e:sublog}. Similarly to~\eqref{eq:omega_x} and~\eqref{e:Omega_prime} we define
\begin{equation}\label{eq:omega_x_bootstrap}
\tilde{\mathcal{G}}_x^{k} = \bigg\{ \mu': \begin{array}{c} \mathcal{B}(x) \subset B(x, h_{\alpha}(n)) \text{ and } |J_{x}(\ell)| \leq \zeta \log n, \\
\text{ for } \ell=h_{\alpha}(n)-1, h_{\alpha}(n), h_{\alpha}(n)+1 \end{array} \bigg\},
\end{equation}
and
\begin{equation}
  \label{e:Omega_prime_bootstrap}
  \Omega(k, n) = \bigcap_{x \in [-n, 4n] \times [-n, 2n]} \tilde{\mathcal{G}}^{k}_x.
\end{equation}
Differently of~\eqref{eq:omega_x}, the event $\tilde{\mathcal{G}}_x^{k}$ controls not only the number of clock ticks in the boundary $\partial B(x, h_{\alpha}(x))$, but also for other values of radii. The bound~\eqref{eq:bound_omega} still holds in this case.

\begin{lemma}
  \label{l:AG_bootstrap}
  Given $k \geq 1$, there exists a constant $\uc{c:AG_bootstrap} = \uc{c:AG_bootstrap}(k) > 0$ such that the following holds for any $\zeta>0$.
  Given $\mu' \in \Omega(k, n)$ and some $t \in [t_-, t_+]$, for every $i$ such that $T_i \leq t$ and $X_i \in [-n, 4n] \times [-n, 2n] \cap \mathbb{Z}^2$ we have
  \begin{equation}
    \label{e:AG_bootstrap}
    \Inf^{\textnormal{thin}} \big( i, \mathcal{C}(3n, n) \big| \mu' \big) \leq \uc{c:AG}^{\zeta \log n} \sum_{x \in \partial B(x_i, h_\alpha(n))} \Inf^{t} \big( x, \mathcal{C}(3n, n) \big| \mu' \big).
  \end{equation}
\end{lemma}

The strategy for proving the lemma above is essentially the same as the one for Lemma~\ref{l:AG}. The main technical difficulty is that we do not pass from a pivotal tick to a pivotal set of initial conditions, but actually to a pivotal set at time $t$.

\begin{proof}
Fix $\mu' \in \Omega(k, n)$, and observe that the occurrence of $\mathcal{C}(3n,n)$ is determined by the evolution restricted to the collection of clock ticks
\begin{equation}
\mathcal{T}= \Big\{j: (X_{j}, T_{j}) \in [-n, 4n]\times [-n, 2n] \times [0,t] \Big\}.
\end{equation}
Assume that the clock tick $i$ is pivotal for a selection of decision variables and clock ticks $\omega = (R_{j},D_{j})_{j \in \mathcal{T}}$, meaning that its selection determines the existence of crossings at time $t$. Observe now that, since $\mu' \in \Omega(k, n)$, the only vertices whose opinions can be changed through the selection of the tick $i$ are the ones in $B(X_{i}, h_{\alpha}(n))$.

As in the proof of Lemma~\ref{l:AG}, we now define a map $\varphi:\big([0,1] \times \{0,1\}\big)^{\mathcal{T}} \to \big([0,1] \times \{0,1\}\big)^{\mathcal{T}}$ that modifies the position of the pivotality. This map will only modify the variables $(R_{j},D_{j})$ for indices in the set
\begin{equation}
J_{X_{i}} =  J_{X_{i}}(h_{\alpha}(n)-1) \cup J_{X_{i}}(h_{\alpha}(n)) \cup J_{X_{i}}(h_{\alpha}(n)+1) \cup \{i\}.
\end{equation}

Given a realization $\omega = (R_{j},D_{j})_{j \in \mathcal{T}}$ for which the clock tick $i$ is pivotal, we split the definition of $\varphi(\omega) = (\tilde{R}_{j},\tilde{D}_{j})_{j \in \mathcal{T}}$ in three steps.
\begin{itemize}
\item First, we set the selection variable $D_{i}$ of the pivotal tick to $\tilde{D}_{i}=1$. This creates open pivotality, meaning that there exists a crossing at time $t$.

\item Second, we change the variables $(R_{j}, D_{j})$ for $j \in J_{X_{i}}(h_{\alpha}(n)-1) \cup J_{X_{i}}(h_{\alpha}(n)+1)$ in order to ensure that the configuration outside $\partial B(X_{i}, h_{\alpha}(n))$ is not modified. For a fixed value of $j$, the modification will depend on whether a flip at $(X_{j},T_{j})$ occurs or not. Assume first that the opinion of vertex $X_{j}$ does not change. In this case, we set $\tilde{D}_{j}=0$ and do not modify $R_{j}$. Whenever a flip occurs, then we set $\tilde{D}_{j}=1$ and $\tilde{R}_{j} = \chi R_{j}$. This last step guarantees that the opinion of this vertex turns to 1 in the new evolution, independently of the neighbors.

\item Finally, we change the variables with $j \in J_{X_{i}}(h_{\alpha}(n))$: freeze all of them by declaring the selection variables $\tilde{D}_{j}=0$.
\end{itemize}

Let us first collect some information about this measure change. First of all, setting $\tilde{D}_{i}=1$ implies that there exists a crossing at time $t$. The second step ensures that the configuration outside  $\partial B(X_{i}, h_{\alpha}(n))$ still remains the same at time $t$. Here we make use of the fact that the rate function $\lambda$ has support of radius one. Finally, pivotality of the clock tick $i$ implies that the configuration after the last modification above does not have any crossings.

In particular, what we obtain from the discussion above is that, if $i$ is a pivotal click for $\omega = (R_{j},D_{j})_{j \in \mathcal{T}}$, then the shell $\partial B(X_{i}, h_{\alpha}(n))$ is closed time-$t$ pivotal for $\varphi(\omega) = (\tilde{R}_{j},\tilde{D}_{j})_{j \in \mathcal{T}}$, meaning that, if all the selection variables of vertices with clock ticks before time $t$ in $\partial B(X_{i}, h_{\alpha}(n))$ are accepted, then the configuration presents a crossing at time $t$.

We want to bound the Radon-Nikodym derivative of $\P \circ \varphi^{-1}$ with respect to $\P$. Since $\varphi$ restricted to $\big([0,1] \times \{0,1\} \big) ^{\mathcal{T} \setminus J_{X_{i}}}$ is the identity, we may restrict ourselves to $\big([0,1] \times \{0,1\} \big) ^{J_{X_{i}}}$. Fix then a set $A=\prod_{j \in J_{X_{i}}}[0, a_{j}] \times \{\tilde{d}_{j}\}$, and notice that
\begin{equation}
\begin{split}
\P\big( \varphi^{-1}(A) \big| \mu' \big) & = \P \Big( \varphi(\omega) \in \prod_{j \in J}[0_{j}, a_{j}] \times \{\tilde{d}_{j}\} \Big| \mu' \Big) \\
& = \sum_{\vec{d} \in \{0,1\}^{J_{X_{i}}}} \P \Big( D_{j}=d_{j}, \tilde{D}_{j} = \tilde{d}_{j}, \tilde{R}_{j} \leq a_{j}, \text{ for all } j \in J_{X_{i}} \Big| \mu' \Big) \\
& \leq \sum_{\vec{d} \in \{0,1\}^{J_{X_{i}}}} \P \Big( \tilde{R}_{j}  \leq a_{j}, \text{ for all } j \in J_{X_{i}} \Big| \mu'\Big).
\end{split}
\end{equation}
In order to bound the probability of the events above, notice that $\tilde{R}_{j} \geq \chi R_{j}$, for every $j \in J_{X_{i}}$. This yields the bound
\begin{equation}
\begin{split}
\P\big( \varphi^{-1}(A) \big| \mu' \big) & \leq \sum_{\vec{d} \in \{0,1\}^{J_{X_{i}}}}\prod_{j \in J_{X_{i}}} \frac{a_{j}}{\chi} \leq 2^{|J_{X_{i}}|}\chi^{-|J_{X_{i}}|} \prod_{j \in J_{X_{i}}} a_{j} \\
& \leq \Big( \frac{2k}{\chi} \Big)^{|J_{X_{i}}|} \P\big( A \big| \mu' \big) \leq \Big( \frac{2k}{\chi} \Big)^{4 \zeta \log n}\P\big( A \big| \mu' \big),
\end{split}
\end{equation}
where the last inequality uses the fact that $|J_{X_{i}}| \leq 3 \zeta \log n+ 1 \leq 4\zeta \log n$, for $n$ large enough.

With this bound on the Radon-Nikodym derivative, and from the properties of our measure transformation, we obtain
\begin{equation}\label{eq:entropy_bootstrap}
\begin{split}
\Inf^{\textnormal{thin}} \big(i, \mathcal{C}(3n, n) \big| \mu' \big) \leq \left(\frac{2k}{\chi}\right)^{4 \zeta \log n}\P\big( \partial B (X_{i}, h_{\alpha}(n)) \text{ is closed time-$t$ pivotal} \big| \mu' \big).
\end{split}
\end{equation}

Finally, it remains to pass from pivotality of the shell $\partial B(X_{i}, h_{\alpha}(n))$ to pivotality of a single site at time $t$. We further modify the clock-tick selections $\tilde{D}_{j}$ in $\partial B(x, h_{\alpha}(n))$ one by one until we eventually arrive in a pivotal vertex. At this point, we stop the changes and do not make any further modifications. This vertex will be closed time-$t$ pivotal. In particular, we obtain the bound
\begin{equation}
\begin{split}
\P\big( \partial B (X_{i}, h_{\alpha}(n)) & \text{ is closed time-$t$ pivotal}\big| \mu' \big) \\
& \qquad \qquad \qquad \leq k^{\zeta \log n}\sum_{x \in \partial B (X_{i}, h_{\alpha}(n))} \Inf^{t} \big( x, \mathcal{C}(3n, n) \big| \mu' \big),
\end{split}
\end{equation}
which concludes the proof, when combined with~\eqref{eq:entropy_bootstrap}.
\end{proof}

The proof of Theorem~\ref{t:phases} can now be adapted to conclude Theorem~\ref{t:phases_bootstrap}.
The changes that have to be made in the proof are:
\begin{itemize}
\item $\rho_+$ and $\rho_-$ should be replaced by $t_+$ and $t_-$ respectively and
\item Lemma~\ref{l:AG} should be replaced by Lemma~\ref{l:AG_bootstrap}.
\end{itemize}
Note that Lemmas \ref{l:quenched_one_arm} and \ref{l:revealment} can be proved with a completely analogous argument as used in Section~\ref{s:proof_ingredients}.
Also the proof that $t_- = t_+ = t_c$ and the bounds \eqref{e:t_small} and \eqref{e:t_large} follow the same arguments as those in the end of the proof of Theorem~\ref{t:phases}.

\section{Open problems and remarks}
\label{s:open_problems}
~
\par This section states several open questions that encourage further research on the field of planar dependent percolation.
  But before presenting these problems, let us quickly mention some of the extensions of our techniques that we believe can be readily obtained without deep new ideas:
\begin{enumerate}
\item Slight variations of the dynamics we considered.
  In this case it is important to make sure that the monotonicity, Harris-FKG inequality, and symmetries of the process still hold true. For example, one could consider Glauber dynamics for the Ising model with nonnegative coupling weights $(J_{x, y})$, invariant under the symmetries of $\Z^2$, and with support in the set of pairs of points which are not too distant from each other.
\item Other lattices besides $\mathbb{Z}^2$ can be also considered, taking care to guarantee that it is still planar in the sense that primal and dual crossings cannot coexist.
  Observe also that invariance with respect to the isometries of $\mathbb{Z}^2$ was an essential ingredient in our proofs.
\item Finite range interactions between spins should also not pose any problem to the presented techniques.
\end{enumerate}

\smallskip

Now that we have mentioned some trivial extensions of our techniques, let us present some more challenging questions that each require a non-trivial insight.
This is a long list and some problems may be significantly more challenging than others.
\begin{enumerate}
\item Can one establish an exponential decay of the connectivity for the off-critical phases? In our Theorem~\ref{t:phases_bootstrap} we fall short of obtaining an actual exponential decay.
\item Can we replace the underlying graph $\mathbb{Z}^2$ with a non-planar counterpart, like slabs of type $\mathbb{Z}^2 \times \{0, \dots, n\}$?
  One of the difficulties in this case would be establishing a RSW theory in the absence of duality.
  A beautiful result in this direction can be found in \cite{DST16}, where the authors establish the absence of percolation for the critical independent model.
\item Our techniques do not work in case spins can interact with unbounded range.
  Even in the case that the coupling constants decay exponentially fast with the distance between sites, new ideas would be necessary to attack sharpness.

\item In order for our argument to work, it was crucial to assume that our dynamic is attractive: more vertices with positive initial states result in more vertices with positive states for later times.
  This was derived from the coupling and the monotonicity of the function $S_\beta$ defined in~\eqref{eq:update_definition}. The proof of our results would work with other functions used to define the dynamics, if they were monotone and depended on finitely many variables. The question that emerges is: how to study the case where no monotonicity is present?
  In~\cite{muirhead2020phase}, the authors study a planar percolation model without positive association, which can give a direction for future research in this vein.

\item Another technique that was used in our proof is the so called ``finite-energy property'', that allows us to change the state of a finite region by paying a bounded price.
  Can one adapt our proofs to models without finite energy, where such surgery constructions are not possible?
  Examples of such models include random interlacements process or the family of limiting distributions of the high-dimensional voter model intersected with the plane.

\item One can investigate sharp thresholds in other related graphs lacking the symmetries of~$\Z^2$, such as book percolation~\cite{duminil2020long} and the quenched measure of percolation with defect lines~\cite{duminil2018brochette}. Since our proof requires these symmetries, the existence of sharp threshold of percolation in these graphs remains open.

\item Studying pivotality in critical and near-critical percolation allows one to prove scaling relations for the universal exponents, should they exist \cite{kesten1987scaling, duminil2020planar}. Can something similar be done to the processes studied here?

\item Our proof also heavily relied on the planarity of the model. An obvious and difficult question is how to prove analogous results for $d \geq 3$~\cite{D-CRT19, D-CGRS20, duminil2020subcritical, D-CRT19b}.

\item The simple algorithm that we have introduced to determine the occurrence of crossings has a small revealment.
  This is enough for us to prove noise sensitivity of the crossing event on \emph{both} initial state variables and the variables used to select which time marks are used by the process.
  Noise sensitivity on each of these variables in isolation is not straightforward though, can one prove it?

\item A challenging question is asking for sharp thresholds of Glauber dynamics \emph{in time}. That is, start the process with a measure out of equilibrium that converges to a measure exhibiting percolation. Is there a time~$\t_c$ neatly dividing a subcritical from a supercritical phase?

\item Consider zero-temperature Glauber dynamics on $\mathbb{Z}^2$ starting with an i.i.d.\ initial condition with initial density $\rho$.
  It is known that for values of $\rho$ close to one, this process stabilizes at positive spins as $\t$ goes to infinity, see \cite{FSS02}.
  By symmetry, the opposite is true for small values of $\rho$.
  A famous folklore conjecture states that there is a phase transition exactly at the value $\rho = 1/2$, see for instance Conjecture~1 in \cite{Morris11}.

\item If we consider the hexagonal lattice, our techniques yield $\rho_c(\t) = 1/2$, for every $\t > 0$. In this case, do we have exceptional times for the percolation event as we let $\t$ grow? This question is related to the previous question about noise sensitivity, see~\cite{schramm2011quantitative, garban2010fourier}.

\item In the divide and color model~\cite{tassion2014planarity} one first samples a subcritical percolation process, and then proceeds to color each finite cluster in one of two colors. This process exhibits asymptotic self-duality in Bernoulli percolation. Can we prove that the divide and color critical parameter in our case also converges to one half?
\end{enumerate}

\appendix

\section*{Appendix - Proof of Lemma~\ref{l:light}}\label{app:lemma_light}

\par In this section, we provide the proof of Lemma~\ref{l:light}. This lemma provides bounds on the probability that a site $y$ is reached by a genealogical path that starts in a far away point $x$.

\begin{proof}[Proof of Lemma~\ref{l:light}]
    Given a sequence of times~$T^{x_0}_{i_0} < \dots < T^{x_k}_{i_k} < t$ realizing the event~$\{ x \overset{t}\rightarrow y \}$, by elementary properties of the Poisson process we have that
    \begin{equation}
      T^{x_{j+1}}_{i_{j+1}} - T^{x_{j}}_{i_{j}} \eqd \Exp(1), \quad
      \text{ for } j = 0, \dots, k-1.
    \end{equation}
    Using a union bound over all possible nearest-neighbor paths from~$x$ to~$y$, we obtain
    \begin{equation}
      \begin{split}
        \mathbb{P} \big[ x \overset{t}\rightarrow y \big]
        &\leq
        \sum_{ \substack{ x_0 \sim \dots \sim x_k \\ x_0 = x, \, x_k = y } }
          \mathbb{P}
          \left[
            \begin{array}{c}
              \text{there exists an associated } \\
              \text{ sequence of Poisson marks } \\
              0 < T^{x_0}_{i_0} < \dots < T^{x_k}_{i_k} < t
            \end{array}
          \right]
        \\
        &\leq
        e^{-t}\sum_{m \geq |x - y|_1} 4^m \sum_{j \geq m} \frac{t^{j}}{j!}.
      \end{split}
    \end{equation}
    Denoting~$|x - y|_1$ by~$n$, we obtain that the right hand side of the last equation is bounded from above by
    \begin{equation}
      \begin{split}
        e^{-t} \sum_{m \geq n}   e^{t} \frac{(4t)^{m}}{m!}
        &\leq
        e^{4t} \frac{(4t)^{n}}{\left(\frac{n}{2}\right)^{\frac{n}{2}}}
        \\
        &\leq
        \exp
          \Big\{
            4t + n\log(4t) - \frac{n\log(n)}{2} + \frac{n}{2}\log (2)
          \Big\}
        \\
        &\leq
        \exp
          \Big\{
            4t -  n\big( \log\big( \sqrt{n} \big)  - \log\big( 4 \sqrt{2} t\big) \big)
          \Big\}.
      \end{split}
    \end{equation}
    Separating into the cases where $n > \big(4\sqrt{2} t \big)^{4}$ and $n \leq \big( 4\sqrt{2} t \big)^{4}$, one sees that the right hand side of the above equation is bounded from above by
    \begin{equation}
      \exp
        \bigg\{
          2^{14}t(1+t^{3}) \log(8t)  -  \frac{1}{4} n \log(n)
        \bigg\}.
    \end{equation}
    This concludes the proof.
  \end{proof}

  \printbibliography
\end{document}